\documentclass{amsart}
\usepackage{amssymb}
\usepackage{amsmath}
\usepackage{amsfonts}

\theoremstyle{plain}

\newtheorem{corollary}{Corollary}

\newtheorem{lemma}{Lemma}
\newtheorem{notation}{Notation}

\newtheorem{proposition}{Proposition}
\newtheorem{remark}{Remark}

\numberwithin{equation}{section}
\input{tcilatex}

\begin{document}
\title{Operator inequalities and characterizations}
\author{A. Seddik}
\address{Department of Mathematics, Faculty of Mathematics and Computer
Science, University of Batna 2, Batna, Algeria.}
\email{a.seddik@univ-batna2.dz}
\urladdr{http://staff.univ-batna2.dz/seddik\_ameur/home?admin\_panel=}
\subjclass{47A30, 47A05, 47B15.}
\keywords{Unitary operator, normal operator, selfadjoint operator,
arithmetic-geometric mean inequality.}

\begin{abstract}
Let $\mathfrak{B}(H)$ be the C*-algebra of all bounded linear operators
acting on a complex separable Hilbert space $H$. In this survey, we shall
present characterizations of some distinguished classes of $\mathfrak{B}(H)$
(namely, normal operators, selfadjoint operators, and unitary operators) in
terms of operator inequalities related to the arithmetic-geometric mean
inequality.

For the class of all normal operators, we shall present new four general
characterizations given as follows: {}

$%
\begin{array}{ccc}
(i) & \left\vert S^{2}\right\vert =\left\vert S\right\vert ^{2},\ \left\vert
S^{\ast 2}\right\vert =\left\vert S^{\ast }\right\vert ^{2}, & (S\in 
\mathfrak{B}(H)), \\ 
(ii) & \left\vert S^{2}\right\vert ^{2}\geq \left\vert S\right\vert ^{4},\
\left\vert S^{\ast 2}\right\vert ^{2}\geq \left\vert S^{\ast }\right\vert
^{4}, & (S\in \mathfrak{B}(H)), \\ 
(iii) & S\ \text{and\ }S^{\ast }\text{ belong to class }\mathbf{A,} & (S\in 
\mathfrak{B}(H)), \\ 
(iv) & S\ \text{and\ }S^{\ast }\text{ are paranormal,} & (S\in \mathfrak{B}%
(H)).%
\end{array}%
$

Note that the characterization $(iv)$ was given by Ando (see $[1]$) but with
the kernel assumption $\ker S=\ker S^{\ast }$.
\end{abstract}

\maketitle

\section{\protect\bigskip Introduction and preliminaries}

Let $\mathfrak{B}(H)$ be the C*-algebra of all bounded linear operators
acting on a complex separable Hilbert space $H$.

We denote by:

\begin{enumerate}
\item[$\bullet $] $u\otimes v$ (where $u,\ v\in H)$, the operator of rank
less or equal to one on $H$ defined by $\left( u\otimes v\right)
x=\left\langle x,v\right\rangle u$, for every $x\in H,$

\item[$\bullet $] $\mathcal{F}_{_{1}}(H)=\left\{ x\otimes y:x,\ y\in
H\right\} $, the set of all operators of rank less or equal to one on $H,$

\item[$\bullet $] $\left\vert S\right\vert ,$ the positive square root of
the positive operator $S^{\ast }S$ (where $S\in \mathfrak{B}(H)),$

\item[$\bullet $] $\left\{ S\right\} ^{^{\prime }}$and$\ \left\{ S\right\}
^{^{\prime \prime }}$, the commutant\ and the bicommutant of $S,$
respectively (where $S\in \mathfrak{B}(H)),$

\item[$\bullet $] $\left( M\right) _{_{1}}=\left\{ x\in M:\left\Vert
x\right\Vert =1\right\} $, for $M$ be a subset of some normed space,

\item[$\bullet $] $K\circ L=\left\{ \sum_{i=1}^{n}\alpha _{_{i}}\beta
_{_{i}}:(\alpha _{_{1}},...,\alpha _{_{n}})\in K,\ (\beta _{_{1}},...,\beta
_{_{n}})\in L\right\} $, for $L,\ K\subset \mathbb{C}^{n}$, $n\geq 1$,

\item[$\bullet $] $\left\vert \Gamma \right\vert =\underset{\gamma \in
\Gamma }{\sup }\left\vert \gamma \right\vert $, where $\Gamma $ is a bounded
subset of the field of scalars,

\item[$\bullet $] $\Gamma M=\left\{ \lambda m:\lambda \in \Gamma ,\ m\in
M\right\} $, where $M$ is a subspace of some vector space, and $\Gamma $ is
a subset of the field of scalars.
\end{enumerate}

For $S,$\ $T\in \mathfrak{B}(H)$:

\begin{enumerate}
\item[$\bullet $] we say that $S$ and $T$ are unitarily equivalent, if there
exists a unitary operator $U\in \mathfrak{B}(H)$ such that $S=U^{\ast }TU$,

\item[$\bullet $] $S$ is paranormal if, $\left\Vert x\right\Vert \left\Vert
S^{2}x\right\Vert \geq \left\Vert Sx\right\Vert ^{2}$, for every $x\in H$,

\item[$\bullet $] we say that $S$ belongs to class $\mathbf{A}$, if $%
\left\vert S^{2}\right\vert \geq \left\vert S\right\vert ^{2},$

\item[$\bullet $] if $S\geq 0,\ T\geq 0$, and $S\geq T$, then $S^{\alpha
}\geq T^{\alpha },$ for every $\alpha \in \lbrack 0,1]$ (L\u{o}wner-Heinz
inequality ($[8]$),

\item[$\bullet $] if $S$ belongs to class $\mathbf{A,}$ then it is
paranormal (see $[6]$).
\end{enumerate}

If $\mathcal{A}$ is a (real or complex) unital normed algebra, and $A\in 
\mathcal{A}$, then

\begin{enumerate}
\item[$\bullet $] we denote by $\sigma (A)$ and $r(A)$, the spectrum and the
spectral radius of $A$, respectively,

\item[$\bullet $] we denote by $V(A)$ and $w(A)$, the algebraic numerical
range and the numerical radius of $A$, respectively,

\item[$\bullet $] $A$ is called convexoid if $V(A)=co\sigma (A)$,

\item[$\bullet $] if $\mathcal{A}=\mathfrak{B}(H)$, then $V(A)=\overline{W(A)%
}$ (where $\overline{W(A)}$ is the closure of the usual numerical range of $%
A $).
\end{enumerate}

For $S\in \mathfrak{B}(H)$, let $R(S)$ and $\ker S$ denote the range and the
kernel of $S$, respectively.

It is known that for $S\in \mathfrak{B}(H)$,

\begin{enumerate}
\item[$\bullet $] then $S$ is of closed range if and only if there exits an
operator $S^{+}\in \mathcal{R}(H)$ satisfying the four following equations%
\begin{equation*}
SS^{+}S=S,\ S^{+}SS^{+}=S^{+},\ (SS^{+})^{\ast }=SS^{+},\ (S^{+}S)^{\ast
}=S^{+}S,
\end{equation*}

\item[$\bullet $] the operator $S^{+}$ if exists is unique, and it is called
the Moore-Penrose inverse of $S$, and it satisfies that $SS^{+}$ and $S^{+}S$
are orthogonal projections onto $R(S)$ and $R(S^{\ast })$, respectively,

\item[$\bullet $] if $S$ is invertible, then $S^{+}=S^{-1}$, and if $S\in 
\mathfrak{B}(H)$ is a surjective operator (resp. injective with closed
range), then $SS^{+}=I$ (rep. $S^{+}S=I).$
\end{enumerate}

For every $S$ in $\mathfrak{B}(H)$ with closed range:

\begin{enumerate}
\item[$\bullet $] we associate the $2\times 2$ matrix representation $S=%
\left[ 
\begin{array}{cc}
S_{_{1}} & S_{_{2}} \\ 
0 & 0%
\end{array}%
\right] $ with respect to the orthogonal direct sum $H=R(S)\oplus \ker
S^{\ast }$,

\item[$\bullet $] the operator $S$ is called an EP operator if $R(S^{\ast
})=R(S)$, or equivalently $S_{_{2}}=0$ and $S_{_{1}}$ is invertible; in this
case $S^{+}=\left[ 
\begin{array}{cc}
S_{_{1}}^{-1} & 0 \\ 
0 & 0%
\end{array}%
\right] $,

\item[$\bullet $] if $S$ is normal, then it is an EP operator.
\end{enumerate}

My main purpose of this survey paper is to present our characterizations
(presented in several papers) of some distinguished classes of $\mathfrak{B}%
(H)$, namely, the selfadjoint operators, the normal operators, and the
unitary operators, in terms of operator inequalities. Our idea is to make a
connection between some operator inequalities related to the known
arithmetic-geometric mean inequality and some remarkable classes of
operators in $\mathfrak{B}(H)$.

\textbf{Part 1. Selfadjoint operators and }$\mathbf{S-AGMI}$\textbf{.}

In $[8]$, Heinz proved that for every two positive operators $P$ and $Q$ in $%
\mathfrak{B}(H)$, and for every $\alpha \in \lbrack 0,1]$, the following
operator inequality holds%
\begin{equation}
\forall X\in \mathfrak{B}(H),\ \left\Vert PX+XQ\right\Vert \geq \left\Vert
P^{\alpha }XQ^{1-\alpha }+P^{1-\alpha }XQ^{\alpha }\right\Vert .  \tag{$HI$}
\end{equation}

As a particular case of this, for $\alpha =\frac{1}{2}$, is the well known
arithmetic-geometric mean inequality given by 
\begin{equation}
\forall A,B,X\in \mathfrak{B}(H),\ \left\Vert A^{\ast }AX+XBB^{\ast
}\right\Vert \geq 2\left\Vert AXB\right\Vert .  \tag{$S-AGMI$}
\end{equation}

Note that the proof of $(HI)$ given by Heinz is somewhat complicated. For
this reason, McIntosh $[11]$ with an elegant proof, proved that the operator
inequality $(S-AGMI)$ holds, and deduced from it the Heinz inequality.

Independently of the work of Heinz and McIntosh, Corach et al. proved in $%
[4] $, that for every invertible selfadjoint operator $S$ in $\mathfrak{B}%
(H) $, the following inequality holds%
\begin{equation}
\forall X\in \mathfrak{B}(H),\ \left\Vert SXS^{-1}+S^{-1}XS\right\Vert \geq
2\left\Vert X\right\Vert ,  \tag{$S1$}
\end{equation}

In $[5,\ 1993]$, Fujji et al. had proved that the three above inequalities
are mutually equivalent, and in $[3,2018]$, it was presented other operator
inequalities that are also equivalent to $(S-AGMI),$ and here we cite two of
them given by:%
\begin{equation}
\forall X\in \mathfrak{B}(H),\ \left\Vert SXS^{+}+S^{+}XS\right\Vert \geq
2\left\Vert SS^{+}XS^{+}S\right\Vert ,  \tag{$S2$}
\end{equation}%
for every selfadjoint operator with closed range $S\in $ $\mathfrak{B}(H),$%
\begin{equation}
\forall X\in \mathfrak{B}(H),\ \left\Vert S^{2}X+XS^{2}\right\Vert \geq
2\left\Vert SXS\right\Vert ,  \tag{$S3$}
\end{equation}%
for every selfadjoint operator $S\in $ $\mathfrak{B}(H).$

Note that $(S2)$ (resp. $(S3)$) is a general form of the Corach-Porta-Recht
inequality from the invertible case to closed range case (resp. the general
situation). This first family of operator inequalities $(S1),\ (S2),\ $and $%
(S3)$ that are equivalent to $(S-AGMI)$ is generated by a selfadjoint
operator (invertible, with closed range, and any).

In $[5]$, Fujji et al. had shown that $\left( S1\right) $ holds with an easy
proof, this gives us an easier proof of Heinz inequality. Without forgetting
that these three inequalities have been generalized from the usual norm to
unitarily invariant norms. For us, we have followed a different way in
introducing in $[12,2001]$ our first characterization of the class of all
invertible operators $S\in \mathfrak{B}(H)$ satisfying the operator
inequality $(S1)$. We have showed that this class is exactly the class of
all invertible selfadjoint operators in $\mathfrak{B}(H)$ multiplied by
nonzero scalars.

Following this kind of problem from the invertible case to the closed range
case, and to the general situation, we have asked to finding the class:

$\mathbf{(i)}\ $of all operators with closed ranges $S\in \mathfrak{B}(H)$
satisfying $(S2),$

$\mathbf{(ii)}$ of all operators $S\in \mathfrak{B}(H)$ satisfying $(S3)$.

We have showed that the class

\begin{enumerate}
\item[$\bullet $]  $\mathbf{(i)}$ is exactly the class of all selfadjoint
operators with closed ranges in $\mathfrak{B}(H)$ multiplied by nonzero
scalars (see $[18,2015]$),

\item[$\bullet $] $\mathbf{(ii)}$ is exactly the class of all selfadjoint
operators in $\mathfrak{B}(H)$ multiplied by nonzero scalars (see $[20,2019]$%
).
\end{enumerate}

\textbf{Part 2. Normal operators and }$\mathbf{N-AGMI}$\textbf{.}

We consider a\ second version of the arithmetic-geometric mean inequality
which follows immediately from $(S-AGMI)$ given as follows:%
\begin{equation}
\forall A,B,X\in \mathfrak{B}(H),\ \left\Vert A^{\ast }AX\right\Vert
+\left\Vert XBB^{\ast }\right\Vert \geq 2\left\Vert AXB\right\Vert 
\tag{$N-AGMI$}
\end{equation}

In $[3,2018]$, we have showed that $(N-AGMI)$ is equivalent to each of the
three following inequalities:

\begin{equation}
\forall X\in \mathfrak{B}(H),\ \left\Vert SXS^{-1}\right\Vert +\left\Vert
S^{-1}XS\right\Vert \geq 2\left\Vert X\right\Vert .  \tag{$N1$}
\end{equation}%
for every invertible normal operator $S\in \mathfrak{B}(H),$

\begin{equation}
\forall X\in \mathfrak{B}(H),\ \left\Vert SXS^{+}\right\Vert +\left\Vert
S^{+}XS\right\Vert \geq 2\left\Vert SS^{+}XS^{+}S\right\Vert ,  \tag{$N2$}
\end{equation}%
for every normal operator with closed range $S\in \mathfrak{B}(H),$%
\begin{equation}
\forall X\in \mathfrak{B}(H),\ \left\Vert S^{2}X\right\Vert +\left\Vert
XS^{2}\right\Vert \geq 2\left\Vert SXS\right\Vert ,  \tag{$N3$}
\end{equation}%
for every normal operators $S\in \mathfrak{B}(H).$

Note that the inequality $(N2)$ (resp. $(N3)$) is a general form of $(N1)$
from the invertible case to the closed range case (resp. the general
situation). In $[3,2018]$, we have proved $(N-AGMI)$ independently to $%
(S-AGMI)$.

This second family of operator inequalities $(N1),\ (N2),\ $and $(N3)$ that
are equivalent to $(N-AGMI)$ is generated by a normal operator (invertible,
with closed range, and any).

As we have done for the characterizations with the above first family, also
it is interesting to describe the largest class

$\mathbf{(i)}\ $of all invertible operators $S\in \mathfrak{B}(H)$
satisfying $(N1),$

$\mathbf{(ii)}\ $of all operators with closed ranges $S\in \mathfrak{B}(H)$
satisfying $(N2),$

$\mathbf{(iii)}$ of all operators $S\in \mathfrak{B}(H)$ satisfying $(N3)$.

We have showed that the class

\begin{enumerate}
\item[$\bullet $] $\mathbf{(i)}$ is exactly the class of all invertible
normal operators in $\mathfrak{B}(H)$ (see $[15,2009]$),

\item[$\bullet $] $\mathbf{(ii)}$ is exactly the class of all normal
operators with closed ranges in $\mathfrak{B}(H)$ (see $[18,2015]$),

\item[$\bullet $] $\mathbf{(iii)}$ is exactly the class of all normal
operators in $\mathfrak{B}(H)$ (see $[20,2019]$).
\end{enumerate}

In this work, we shall present four new general characterizations of the
class of all normal operators in $\mathfrak{B}(H),$ that are given as
follows:

$\mathbf{(i)\ }\left\vert S^{2}\right\vert =\left\vert S\right\vert ^{2},\
\left\vert S^{\ast 2}\right\vert =\left\vert S^{\ast }\right\vert ^{2},\ \
(S\in \mathfrak{B}(H)),$

$\mathbf{(ii)}\ \left\vert S^{2}\right\vert ^{2}\geq \left\vert S\right\vert
^{4},\ \left\vert S^{\ast 2}\right\vert ^{2}\geq \left\vert S^{\ast
}\right\vert ^{4},\ \ \ (S\in \mathfrak{B}(H)),$

$\mathbf{(iii)}\ S$ and $S^{\ast }$ belong to class $\mathbf{A,}$ \ \ $(S\in 
\mathfrak{B}(H)),$

$\mathbf{(iv)}\ S$ and $S^{\ast }$ are paranormal, \ \ $(S\in \mathfrak{B}%
(H)).$

Note that the characterization $\mathbf{(iv)}$ was given by T. Ando $[1,\
1972]$, but with the kernel assumption $\ker S=\ker S^{\ast }$.

\textbf{Part 3. Unitary operators.}

Let $S$ be an invertible operator in $\mathfrak{B}(H)$.

It is clear that 
\begin{equation*}
\underset{\left\Vert X\right\Vert =1}{\inf }\left\Vert
SXS^{-1}+S^{-1}XS\right\Vert \leq 2\leq \underset{\left\Vert X\right\Vert =1}%
{\sup }\left\Vert SXS^{-1}+S^{-1}XS\right\Vert .
\end{equation*}

Then from the above first part with the invertible case, the above infimum
gets its maximal value $2$ if and only if $S$ is an invertible selfadjoint
operator in $\mathfrak{B}(H)$ multiplied by a nonzero scalar.

So, what about the operator $S$ for which the above supremum gets its
minimal value $2$? It was proved that:

$\mathbf{(i)}$ this supremum gets its minimal value $2$ if and only if $S$
is a unitary operator multiplied by a nonzero scalar (see $[16,\ 2011]$),

$\mathbf{(ii)}\ \underset{\left\Vert X\right\Vert =1=rankX}{\sup }\left\Vert
SXS^{-1}+S^{-1}XS\right\Vert \geq 2$ (see $[15,\ 2009]$),

$\mathbf{(iii)}$ this last supremum gets its minimal value $2$ if and only
if $S$ is normal and $\underset{\lambda ,\mu \in \sigma (S)}{\sup }%
\left\vert \frac{\lambda }{\mu }+\frac{\mu }{\lambda }\right\vert =2$ (see $%
[16,\ 2011]$).

From $\mathbf{(i)}$, the class of all unitary operators in $\mathfrak{B}(H)$
multiplied by nonzero scalars is exactly the class of all invertible
operators $S\in $ $\mathfrak{B}(H)$ satisfying the following operator
inequality: 
\begin{equation*}
\forall X\in \mathfrak{B}(H),\ \left\Vert SXS^{-1}+S^{-1}XS\right\Vert \leq
2\left\Vert X\right\Vert .
\end{equation*}

From $\mathbf{(iii)}$, the class of all invertible normal operators $S\in 
\mathfrak{B}(H)$ for which $\underset{\lambda ,\mu \in \sigma (S)}{\sup }%
\left\vert \frac{\lambda }{\mu }+\frac{\mu }{\lambda }\right\vert =2$ is
exactly the class of all invertible operators $S\in \mathfrak{B}(H)$
satisfying the following operator inequality: 
\begin{equation*}
\forall X\in \mathcal{F}_{_{1}}(H),\ \left\Vert SXS^{-1}+S^{-1}XS\right\Vert
\leq 2\left\Vert X\right\Vert .
\end{equation*}

This second class contains strictly the first class and contained strictly
in the class of all invertible normal operators in $\mathfrak{B}(H)$.

\textbf{Part 4. Some comments.}

Note that for the characterization concerning $(S2)$ (resp. $(N2)$) with the
closed range case is deduced from the characterization of $(S1)$ (resp. $%
(N1) $), and the characterization concerning $(S3)$ (resp. $(N3)$) with the
general situation is deduced from the characterization of $(S2)$ (resp. $%
(N2) $).

Unfortunately, after the publication of the paper $[18]$ concerning the
closed range case, we have found a mistake in $Lemma\ 1$, and all results
concerning the closed range case (for the two above parts) depend on it. So,
in the corrigendum $[19,\ 2017]$, we have presented a corrected proof of
this lemma. Note that in the proof of this corrected lemma, we have used the 
$Theorem\ 3.6$ of $[12]$, where one of the conditions of this theorem is an
equality between spectrum of two positive operators. This condition is
enough for the invertible case only, but not suffice for non-invertible case
and our lemma is for non-invertible case. But, to have a complete proof of
the lemma, we need $Theorem\ 6.3$ with inclusion between spectrum instead of
equality. We have mentioned in the proof of the corrected lemma that $%
Theorem\ 6.3$ remains true with inclusion between spectrum but without
argument. In this survey, we shall present this argument.

This work gives to the separately published results a certain harmony and
concordance. The starting point in this work will be the largest class of
normal operators, then their two subclasses of self-adjoint operators and
unitary operators, and will end with the intersection of these two last
subclasses, namely, the class of unitary reflections. We will then notice
that each class will be linked to an operator inequality and how their forms
vary from one to the other.

In section 2, we shall present a family of operator inequalities that are
equivalent to $(N-AGMI),\ $and the characterizations cited above in the part
2 concerning the normal operators.

In section 3, we shall present a family of operator inequalities that are
equivalent to $(S-AGMI),$ and the characterizations cited above in the part
1 concerning the selfadjoint operators.

In section 4, we shall present some results concerning the injective norm of
the two following operators on $\mathfrak{B}(H)$: 
\begin{equation*}
X\rightarrow SXS^{-1}+S^{-1}XS,\ \ X\rightarrow S^{\ast
}XS^{-1}+S^{-1}XS^{\ast }
\end{equation*}%
(where $S$ be an invertible operator in $\mathfrak{B}(H)$), and some
characterizations that are cited in the part 3 and others.

\section{N-Arithmetic-Geometric-Mean Inequality, Normal operators, and
Characterizations}

In this section, we shall present some characterizations of the class of all
normal operators in $\mathfrak{B}(H)$ in terms of operator inequalities, and
also its two subclasses of all invertible normal operators and all normal
operators with closed ranges in $\mathfrak{B}(H).$ These operator
inequalities are related to the N-Arithmetic-Geometric-Mean Inequality which
will be introduced in the first proposition.

We start with the following remark which contains two trivial
characterizations of the class of all normal operators in $\mathfrak{B}(H)$.

\begin{remark}
Let $S\in \mathfrak{B}(H)$. It is easy to see that the three following
properties are equivalent:

$(i)$ $S$ is normal,

$(ii)\ \forall X\in \mathfrak{B}(H),\ \left\Vert S^{\ast }X\right\Vert
=\left\Vert SX\right\Vert ,$

$(iii)\ \forall X\in \mathfrak{B}(H),\ \left\Vert XS^{\ast }\right\Vert
=\left\Vert XS\right\Vert .$
\end{remark}

In this section, we consider the N-Arithmetic-Geometric Mean Inequality
given by: 
\begin{equation}
\forall A,B,X\in \mathfrak{B}(H),\ \left\Vert A^{\ast }AX\right\Vert
+\left\Vert XBB^{\ast }\right\Vert \geq 2\left\Vert AXB\right\Vert . 
\tag{$N-AGMI$}
\end{equation}

This inequality follows immediately from the known Arithmetic-Geometric-Mean
Inequality (which is called here $\left( S-AGMI\right) $). In the next
proposition, we present a family of operator inequalities generated by
normal operators that are equivalent to the $\left( N-AGMI\right) $, and \
we shall prove $\left( N-AGMI\right) $ independently on $\left(
S-AGMI\right) $.

\begin{proposition}
$[3]\ $The following operator inequalities hold and are mutually equivalent:%
\begin{equation}
\forall X\in \mathfrak{B}(H),\ \left\Vert A^{\ast }AX\right\Vert +\left\Vert
XBB^{\ast }\right\Vert \geq 2\left\Vert AXB\right\Vert ,  \tag{$1$}
\end{equation}%
for every $A,\ B\in \mathfrak{B}(H)$, 
\begin{equation}
\forall X\in \mathfrak{B}(H),\ \left\Vert SXR^{+}\right\Vert +\left\Vert
S^{+}XR\right\Vert \geq 2\left\Vert SS^{+}XR^{+}R\right\Vert ,  \tag{$2$}
\end{equation}%
for every normal operators with closed ranges $S,R\in \mathfrak{B}(H),$ 
\begin{equation}
\forall X\in \mathfrak{B}(H),\ \left\Vert SXR^{-1}\right\Vert +\left\Vert
S^{-1}XR\right\Vert \geq 2\left\Vert X\right\Vert ,  \tag{$3$}
\end{equation}%
for every invertible normal operators $S,R\in \mathfrak{B}(H),$%
\begin{equation}
\forall X\in \mathfrak{B}(H),\left\Vert S^{2}X\right\Vert +\left\Vert
XR^{2}\right\Vert \geq 2\left\Vert SXR\right\Vert ,  \tag{$4$}
\end{equation}%
for every normal operators $S,R\in \mathfrak{B}(H),$%
\begin{equation}
\forall X\in \mathfrak{B}(H),\ \left\Vert A^{\ast }AX\right\Vert +\left\Vert
XAA^{\ast }\right\Vert \geq 2\left\Vert AXA\right\Vert ,  \tag{$1^{\prime }$}
\end{equation}%
for every $A\in \mathfrak{B}(H),$ 
\begin{equation}
\forall X\in \mathfrak{B}(H),\ \left\Vert SXS^{+}\right\Vert +\left\Vert
S^{+}XS\right\Vert \geq 2\left\Vert SS^{+}XS^{+}S\right\Vert , 
\tag{$2^{\prime }$}
\end{equation}%
for every normal operator with closed range $S\in \mathfrak{B}(H),$%
\begin{equation}
\forall X\in \mathfrak{B}(H),\ \left\Vert SXS^{-1}\right\Vert +\left\Vert
S^{-1}XS\right\Vert \geq 2\left\Vert X\right\Vert ,  \tag{$3^{\prime }$}
\end{equation}%
for every invertible normal operator $S\in \mathfrak{B}(H),$%
\begin{equation}
\forall X\in \mathfrak{B}(H),\ \left\Vert S^{2}X\right\Vert +\left\Vert
XS^{2}\right\Vert \geq 2\left\Vert SXS\right\Vert ,  \tag{$4^{\prime }$}
\end{equation}%
for every normal operator $S\in \mathfrak{B}(H),$
\end{proposition}

\begin{proof}
$(1)\Rightarrow (2).$ Assume $(1)$ holds. Let $S,R$ be two normal operators
with closed ranges in $\mathfrak{B}(H)$, and let $X\in \mathfrak{B}(H)$.
Since $S^{\ast }=S^{\ast }SS^{+}$ and $R^{\ast }=R^{+}RR^{\ast }$, then from 
$(1)$ and $(Remark\ 1)$, it follows that 
\begin{eqnarray*}
\left\Vert SXR^{+}\right\Vert +\left\Vert S^{+}XR\right\Vert &=&\left\Vert
S^{\ast }S\left( S^{+}XR^{+}\right) \right\Vert +\left\Vert \left(
S^{+}XR^{+}\right) RR^{\ast }\right\Vert \\
&\geq &2\left\Vert SS^{+}XR^{+}R\right\Vert .
\end{eqnarray*}

Hence $(2)$ holds.

$(2)\Rightarrow (3).$ This implication is trivial.

$(3)\Rightarrow (4).$ Assume $(3)$ holds.

let $S,\ R,\ X\in \mathfrak{B}(H)$ such that $S,\ R$ are normal. Put $%
P=\left\vert S\right\vert ,\ \ Q=\left\vert R^{\ast }\right\vert $, and let $%
\epsilon >0$.

It is clear that the two operators $P+\epsilon I$ and $Q+\epsilon I$ are
normal and invertible. So, from $(3)$, we obtain%
\begin{equation*}
\forall \epsilon >0,\ \forall X\in \mathfrak{B}(H),\ \left\Vert \left(
P+\epsilon I\right) ^{2}X\right\Vert +\left\Vert X\left( Q+\epsilon I\right)
^{2}\right\Vert \geq 2\left\Vert \left( P+\epsilon I\right) X\left(
Q+\epsilon I\right) \right\Vert .
\end{equation*}

By letting $\epsilon \rightarrow 0$, we deduce that:%
\begin{equation*}
(\ast )\ \ \ \forall X\in \mathfrak{B}(H),\ \left\Vert P^{2}X\right\Vert
+\left\Vert XQ^{2}\right\Vert \geq 2\left\Vert PXQ\right\Vert .
\end{equation*}

So, we have: 
\begin{eqnarray*}
\left\Vert S^{2}X\right\Vert +\left\Vert XR^{2}\right\Vert &=&\left\Vert
S^{\ast }SX\right\Vert +\left\Vert XRR^{\ast }\right\Vert \text{ \ \ (from }%
Remark\ 1\text{), } \\
&=&\left\Vert P^{2}X\right\Vert +\left\Vert XQ^{2}\right\Vert \text{ \ } \\
&\geq &2\left\Vert PXQ\right\Vert \text{ \ \ \ (from }(\ast )\text{),} \\
&=&2\left\Vert SXR\right\Vert .
\end{eqnarray*}

This proves $(4)$.

$(4)\Rightarrow (1)$. Assume $(4)$ holds.

Let $A,\ B,\ X\in \mathfrak{B}(H)$. Put $P=\left\vert A\right\vert ,\
Q=\left\vert B^{\ast }\right\vert $. Then, we have:%
\begin{eqnarray*}
\left\Vert A^{\ast }AX\right\Vert +\left\Vert XBB^{\ast }\right\Vert
&=&\left\Vert P^{2}X\right\Vert +\left\Vert XQ^{2}\right\Vert \\
&\geq &2\left\Vert PXQ\right\Vert \text{ \ \ (from }(4)\text{),} \\
&=&2\left\Vert AXB\right\Vert .
\end{eqnarray*}

This proves $(1)$.

Therefore the operator inequalities $(1)-(4)$ are equivalent.

From a pair of operators to a single operator, we deduce that the operator
inequalities $(1^{\prime })-(4^{\prime })$ are also equivalent.

$(1)\Rightarrow (1^{\prime })$. This implication is trivial.

$(1^{\prime })\Rightarrow (1)$. Assume $(1^{\prime })$ holds (here we use
the Berberian technic).

Let $A,B,X\in \mathfrak{B}(H)$. Consider now, the bounded linear operators $%
C,\ Y$ defined on the Hilbert space $H\oplus H$ given by $C=\left[ 
\begin{array}{cc}
A & 0 \\ 
0 & B%
\end{array}%
\right] ,\ Y=\left[ 
\begin{array}{cc}
0 & X \\ 
0 & 0%
\end{array}%
\right] $. By a simple computation, we obtain $C^{\ast }CY=\left[ 
\begin{array}{cc}
0 & A^{\ast }AX \\ 
0 & 0%
\end{array}%
\right] ,\ YCC^{\ast }=\left[ 
\begin{array}{cc}
0 & XBB^{\ast } \\ 
0 & 0%
\end{array}%
\right] $, and $CYC=\left[ 
\begin{array}{cc}
0 & AXB \\ 
0 & 0%
\end{array}%
\right] $. Applying $(1^{\prime })$ for the Hilbert space $H\oplus H$, we
obtain $\left\Vert A^{\ast }AX\right\Vert +\left\Vert XAA^{\ast }\right\Vert
=\left\Vert C^{\ast }CY\right\Vert +\left\Vert YCC^{\ast }\right\Vert \geq
2\left\Vert CYC\right\Vert =2\left\Vert AXB\right\Vert $. This proves $(1)$.

Therefore the inequalities $(1)-(4)$, and $(1^{\prime })-(4^{\prime })$ are
mutually equivalent. It remains to prove that one of them holds. It is clear
that $(1)$ is an immediate consequence of the known Arithmetic-geometric
mean inequality $(S-AGMI)$. But here, we shall give a direct proof of $(1)$
independently of $(S-AGMI)$ by using the numerical arithmetic-geometric mean
inequality. Let $A,B,X\in \mathfrak{B}(H).$ The following inequalities hold:%
\begin{eqnarray*}
\frac{1}{2}\left( \left\Vert A^{\ast }AX\right\Vert +\left\Vert XBB^{\ast
}\right\Vert \right) &\geq &\sqrt{\left\Vert A^{\ast }AX\right\Vert
\left\Vert XBB^{\ast }\right\Vert } \\
&\geq &\sqrt{\left\Vert BB^{\ast }X^{\ast }A^{\ast }AX\right\Vert } \\
&\geq &\sqrt{r(BB^{\ast }X^{\ast }A^{\ast }AX)} \\
&=&\sqrt{r\left( B^{\ast }X^{\ast }A^{\ast }AXB\right) } \\
&=&\left\Vert AXB\right\Vert .
\end{eqnarray*}
\end{proof}

\begin{corollary}
The following operator inequalities hold and are equivalent to $(N-AGMI)$: 
\begin{equation}
\forall X\in \mathfrak{B}(H),\ \left\Vert S^{\ast }XR^{+}\right\Vert
+\left\Vert S^{+}XR^{\ast }\right\Vert \geq 2\left\Vert
SS^{+}XR^{+}R\right\Vert ,  \tag{$5$}
\end{equation}%
for every operators with closed ranges $S,R\in \mathfrak{B}(H)$,

\begin{equation}
\forall X\in \mathfrak{B}(H),\ \left\Vert S^{\ast }XR^{-1}\right\Vert
+\left\Vert S^{-1}XR^{\ast }\right\Vert \geq 2\left\Vert X\right\Vert , 
\tag{$6$}
\end{equation}%
for every invertible operators $S,R\in \mathfrak{B}(H)$,

\begin{equation}
\forall X\in \mathfrak{B}(H),\ \left\Vert S^{\ast }XS^{+}\right\Vert
+\left\Vert S^{+}XS^{\ast }\right\Vert \geq 2\left\Vert
SS^{+}XS^{+}S\right\Vert ,  \tag{$5^{\prime }$}
\end{equation}%
for every operator with closed range $S\in \mathfrak{B}(H)$,

\begin{equation}
\forall X\in \mathfrak{B}(H),\ \left\Vert S^{\ast }XS^{-1}\right\Vert
+\left\Vert S^{-1}XS^{\ast }\right\Vert \geq 2\left\Vert X\right\Vert , 
\tag{$6^{\prime }$}
\end{equation}%
for every invertible operator $S\in \mathfrak{B}(H)$.
\end{corollary}

\begin{proof}
Assume $(N-AGMI)$ holds. Prove that $(5)$ holds.

Let $S,R\in \mathcal{R}(H)$, and $X\in \mathfrak{B}(H)$. Since, $SS^{+}S=S$
and $RR^{+}R=R$, then we have 
\begin{eqnarray*}
\left\Vert S^{\ast }XR^{+}\right\Vert +\left\Vert S^{+}XR^{\ast }\right\Vert
&=&\left\Vert S^{\ast }S\left( S^{+}XR^{+}\right) \right\Vert +\left\Vert
\left( S^{+}XR^{+}\right) RR^{\ast }\right\Vert , \\
&\geq &2\left\Vert SS^{+}XR^{+}R\right\Vert \text{, \ \ \ (from }(N-AGMI)%
\text{).}
\end{eqnarray*}

This proves $(5)$.

It is clear that $(5)$ implies $(6),\ (5^{\prime }),\ (6^{\prime })$, and
using $Remark\ 1$, then $(5)$ (resp. $(6),\ (5^{\prime }),\ (6^{\prime })$)
implies $(2)$ (resp. $(3),\ (2^{\prime }),\ (3^{\prime })$).
\end{proof}

Note that the six operator inequalities $(2)-(4)$ and $(2^{^{\prime
}})-(4^{^{\prime }})$ given in the last proposition are generated by a pair
of normal operators and a single of normal operator, respectively.

We shall interest to describe the class of

$\mathbf{(i)}\ $all invertible operators $S\in \mathfrak{B}(H)$ satisfying \
the operator inequality $(3^{^{\prime }}),$

$\mathbf{(ii)}\ $all operators with closed ranges $S\in \mathfrak{B}(H)$
satisfying \ the operator inequality $(2^{^{\prime }}),$

$\mathbf{(iii)}\ $all operators $S\in \mathfrak{B}(H)$ satisfying the
operator inequality $(4^{^{\prime }}).$

We shall prove that the class

\begin{enumerate}
\item[$\bullet $] $\mathbf{(i)}$ is the class of all invertible normal
operators in $\mathfrak{B}(H),$

\item[$\bullet $] $\mathbf{(ii)}$ is the class of all normal operators with
closed ranges in $\mathfrak{B}(H),$

\item[$\bullet $] $\mathbf{(iii)}$\textbf{\ }is the class of all normal
operators in $\mathfrak{B}(H).$
\end{enumerate}

We shall present\ here all these characterizations and others. We need \ the
following lemmas.

\begin{lemma}
$[22]$. Let $A\in \mathfrak{B}(H)$. If $\left\Vert A-\lambda I\right\Vert
=r(A-\lambda I)$, for all complex$\lambda $, then $A$ is convexoid.
\end{lemma}

\begin{lemma}
$[12]\ $Let $P,\ Q$ be two invertible positive operators in $\mathfrak{B}(H)$
satisfying the following operator inequality%
\begin{equation*}
\forall X\in \mathfrak{B}(H),\ \left\Vert X\right\Vert +\left\Vert
PXP^{-1}\right\Vert \geq 2\left\Vert QXQ^{-1}\right\Vert .
\end{equation*}

Then, we have $\left\{ P\right\} ^{,}\subset \left\{ Q\right\} ^{,}.$
\end{lemma}

\begin{proof}
$(i).$ Let $X$ be a selfadjoint operator in $\mathfrak{B}(H)$ such that $%
PX=XP$, and let $\alpha $ be an arbitrary complex number. Replace $X$ by $%
X-\alpha I$ in the inequality given by the lemma, and since $X-\alpha I$ is
normal, we obtain 
\begin{equation*}
\left\Vert X-\alpha I\right\Vert \geq \left\Vert Q(X-\alpha
I)Q^{-1}\right\Vert \geq r\left( Q(X-\alpha I)Q^{-1}\right) =\left\Vert
X-\alpha I\right\Vert .
\end{equation*}

Hence, $\left\Vert QXQ^{-1}-\alpha I\right\Vert =r\left( QXQ^{-1}-\alpha
I\right) $, for all complex number $\alpha $. Using the above lemma, we
obtain that the 
\begin{eqnarray*}
V\left( QXQ^{-1}\right) &=&co\left( \sigma \left( QXQ^{-1}\right) \right) ,
\\
&=&co\sigma \left( X\right) , \\
&\subset &\mathbb{R}.
\end{eqnarray*}%
This give us that $QXQ^{-1}$ is selfadjoint. Hence, $QX=XQ$.

$(ii).$ Now, let $X$ be an arbitrary operator in $\mathfrak{B}(H)$. Put $%
X=X_{1}+iX_{2}$, where $X_{1}=\func{Re}(X)$, and $X_{2}=\func{Im}X$. Assume
that $PX=XP$. Then, $PX_{1}=X_{1}P$ and $PX_{2}=X_{2}P$.\ From $(i)$, we
deduce that, $QX_{1}=X_{1}Q$ and $QX_{2}=X_{2}Q$. Thus, $QX=XQ$. Therefore, $%
\left\{ P\right\} ^{^{,}}\subset \left\{ Q\right\} ^{^{,}}$.
\end{proof}

\begin{lemma}
$[12]\ $Let $P,\ Q$ be two invertible positive operators in $\mathfrak{B}(H)$
satisfying the following operator inequality%
\begin{equation*}
\forall X\in \mathfrak{B}(H),\ \left\Vert PXP^{-1}\right\Vert +\left\Vert
Q^{-1}XQ\right\Vert \geq 2\left\Vert X\right\Vert .
\end{equation*}

Then, we have $\left\{ P\right\} ^{^{,}}=\left\{ Q\right\} ^{^{,}}$.
\end{lemma}

\begin{proof}
From the inequality given in the lemma, we have%
\begin{equation*}
\forall X\in \mathfrak{B}(H),\ \left\Vert X\right\Vert +\left\Vert
PQXQ^{-1}P^{-1}\right\Vert \geq 2\left\Vert QXQ^{-1}\right\Vert \ (\ast ).
\end{equation*}

Put $M=\left\vert PQ\right\vert $. So, from this last inequality, we obtain

\begin{equation*}
\forall X\in \mathfrak{B}(H),\ \left\Vert X\right\Vert +\left\Vert
MXM^{-1}\right\Vert \geq 2\left\Vert QXQ^{-1}\right\Vert .
\end{equation*}

Hence, from the above lemma, we deduce that, $MQ=QM$. Then, $PQ=QP$.

Now, let $X$ be a selfadjoint operator in $\mathfrak{B}(H)$ such that $PX=XP$%
, and let $\alpha $ be an arbitrary complex number. Replace in $(\ast )$, $X$
by $X-\alpha I$, so we have, $\left\Vert X-\alpha I\right\Vert \geq
\left\Vert Q(X-\alpha I)Q^{-1}\right\Vert $, for every complex number $%
\alpha $. hence, $QX=XQ$, and \ thus $\left\{ P\right\} ^{^{,}}\subset
\left\{ Q\right\} ^{^{,}}$, this follows by using the same argument as used
in the proof of the above lemma.

Using again the inequality given in the lemma, we obtain also that $\left\{
Q\right\} ^{^{,}}=\left\{ Q^{-1}\right\} ^{^{,}}\subset \left\{
P^{-1}\right\} ^{\prime }=\left\{ P\right\} ^{^{,}}$. Therefore, $\left\{
P\right\} ^{^{,}}=\left\{ Q\right\} ^{^{,}}$.
\end{proof}

\begin{lemma}
Let $\epsilon >0$, and let $\alpha _{_{1}},...,\alpha _{_{n}},\ \beta
_{_{1}},...,\beta _{_{n}}\ ($where $n\geq 1)$ be positive real numbers such
that $0<\alpha _{_{1}}\leq ...\leq \alpha _{_{n}}\leq 1$, $\left\{ \alpha
_{_{1}},...,\alpha _{_{n}}\right\} \subset \left\{ \beta _{_{1}},...,\beta
_{_{n}}\right\} $, and $\frac{\alpha _{_{i}}}{\alpha _{j}}+\frac{\beta
_{_{j}}}{\beta _{i}}\geq 2-\epsilon $, for every $i,\ j$.\ Then, we have $%
\left\vert \alpha _{_{i}}-\beta _{_{i}}\right\vert \leq \epsilon $, for $%
i=1,...,n$.
\end{lemma}

\begin{proof}
From the hypothesis, we deduce easily that $\beta _{i}-\beta _{_{j}}\leq
\epsilon $, if $i<j.$

Let $i\in \left\{ 1,...,n\right\} $ such that $\alpha _{_{i}}\neq \beta
_{_{i}}$ (in the case $\alpha _{_{i}}=\beta _{_{i}}$, of course we have $%
\left\vert \alpha _{_{i}}-\beta _{_{i}}\right\vert =0\leq \epsilon $).

There are three cases: $i=1,\ i=n,\ $and $1<i<n.$

Case 1. $i=1$. There exists $j\geq 2$ such that $\beta _{_{j}}=\alpha
_{_{1}} $. So, we have, $\left\vert \alpha _{1}-\beta _{1}\right\vert =\beta
_{1}-\beta _{j}\leq \epsilon $, since $j>1$.

Case 2.$\ i=n.\ $There exists $j<n$ such that $\beta _{_{j}}=\alpha _{_{n}}$%
. Hence, $\left\vert \alpha _{n}-\beta _{n}\right\vert =\beta _{j}-\beta
_{n}\leq \epsilon $, since $j<n$.

Case 3. $1<i<n.$

If $\alpha _{_{i}}<\beta _{_{i}}$, then there exists $j>i$, such that $\beta
_{_{j}}\leq \alpha _{_{i}}$.\ Hence, $\left\vert \alpha _{_{i}}-\beta
_{_{i}}\right\vert \leq \beta _{i}-\beta _{_{j}}\leq \epsilon $, since $i<j.$

If $\alpha _{_{i}}>\beta _{_{i}}$, then there exists $j<i$, such that $\beta
_{_{j}}\geq \alpha _{_{i}}$.\ Hence, $\left\vert \alpha _{_{i}}-\beta
_{_{i}}\right\vert \leq \beta _{j}-\beta _{i}\leq \epsilon $, since $i>j.$
\end{proof}

\begin{remark}
In the original paper $[12]$, the above lemma is $Lemma\ 3.5$, but it is
given in a particular case with equality instead of inclusion, and where the
sequence $\left\{ \alpha _{_{1}},...,\alpha _{_{n}}\right\} $ is increasing
instead of non-decreasing; but the proof remains unchanged. In the
particular case, the $Lemma\ 3.5$ is needed only for invertible case.
\end{remark}

\begin{lemma}
Let $P,\ Q$ be two invertible positive operators in $\mathfrak{B}(H)$ such
that $\sigma (Q)\subset \sigma (P)$ or $\sigma (P)\subset \sigma (Q)$. Then,
the two following properties are equivalent

$(i)\ \forall X\in \mathfrak{B}(H),\ \left\Vert PXP^{-1}\right\Vert
+\left\Vert Q^{-1}XQ\right\Vert \geq 2\left\Vert X\right\Vert ,$

$(ii)\ P=Q.$
\end{lemma}

\begin{proof}
We may assume without loss of the generality that $\left\Vert P\right\Vert
=\left\Vert Q\right\Vert =1.$

$(i)\Rightarrow (ii).\ $Assume $(i)$ holds.

Decompose $P$ and $Q$ using their spectral measure 
\begin{equation*}
P=\dint \lambda dE_{_{\lambda }},\ \ \ Q=\dint \lambda dF\lambda
\end{equation*}%
and consider 
\begin{equation*}
P_{_{n}}=\dint h_{_{n}}\left( \lambda \right) dE_{_{\lambda
}}=h_{_{n}}\left( P\right) ,\ \ \ Q_{_{n}}=\dint h_{_{n}}\left( \lambda
\right) dF\lambda =h_{_{n}}\left( Q\right) .
\end{equation*}%
$\ $where $h_{_{n}}\left( \lambda \right) $ is the function defined by 
\begin{equation*}
h_{_{n}}\left( \lambda \right) =\frac{k}{n},\text{if }\frac{k}{n}\leq
\lambda <\frac{k+1}{n},\text{ for }k=1,\ 2,\ 3,...
\end{equation*}

Case 1. $\sigma (Q)\subset \sigma (P).$ Using the spectral theorem with the
function $h_{_{n}}$, we have 
\begin{equation*}
\sigma \left( Q_{_{n}}\right) =\sigma \left( h_{_{n}}\left( Q\right) \right)
=h_{_{n}}\left( \sigma (Q)\right) \subset h_{_{n}}\left( \sigma (P)\right)
=\sigma \left( h_{_{n}}\left( P\right) \right) =\sigma \left( P_{_{n}}\right)
\end{equation*}

Then $P_{_{n}}$,$\ Q_{_{n}}$ are invertible positive operators in $\mathfrak{%
B}(H)$ with finite spectrum such that $\sigma \left( Q_{_{n}}\right) \subset
\sigma \left( P_{_{n}}\right) $, $P_{_{n}}\rightarrow P,\
Q_{_{n}}\rightarrow Q$ uniformly, and $P_{_{n}}\in \left\{ P\right\}
^{^{\prime \prime }},\ Q_{_{n}}\in \left\{ Q\right\} ^{^{\prime \prime }}$
(where $\left\{ P\right\} ^{^{\prime \prime }}=\left\{ Q\right\} ^{^{\prime
\prime }}$, from the $Lemma\ 3$).

Hence, $P_{_{n}}Q_{_{n}}=Q_{_{n}}P_{_{n}}$, for every $n\geq 1$. Then, $%
Q_{_{n}}=\sum_{i=}^{p}\alpha _{_{i}}E_{i}$, $P_{_{n}}=\sum_{i=}^{p}\beta
_{_{i}}E_{i}$, where $\sigma \left( Q_{_{n}}\right) =\left\{ \alpha
_{_{1}},..,\alpha _{_{p}}\right\} $ such that $0\,<\alpha _{_{1}}\leq
...\leq \alpha _{_{p}}\leq 1$, $\sigma \left( P_{_{n}}\right) =\left\{ \beta
_{_{1}},..,\beta _{_{p}}\right\} $, and $E_{_{1}},...,E_{_{p}}$ are
orthogonal projections in $\mathfrak{B}(H)$ such that $E_{_{i}}E_{_{j}}=0$,
if $i\neq j$, $\ \sum_{i=1}^{p}E_{_{i}}=I.$ Thus, $\left\{ \alpha
_{_{1}},..,\alpha _{_{p}}\right\} \subset \left\{ \beta _{_{1}},..,\beta
_{_{p}}\right\} .$

Let $\epsilon >0$. Then, there exists an integer $N\geq 1$ such that 
\begin{equation*}
\forall n>N,\ \forall X\in \mathfrak{B}(H),\ \left\Vert
P_{_{n}}XP_{_{n}}^{-1}\right\Vert +\left\Vert
Q_{_{n}}^{-1}XQ_{_{n}}\right\Vert \geq \left( 2-\epsilon \right) \left\Vert
X\right\Vert .
\end{equation*}

Let $n>N$, and replace $X$ by $E_{_{i}}XE_{_{j}}$ (where $X\in \mathfrak{B}%
(H)$) in this last inequality$,$ we deduce that 
\begin{equation*}
\frac{\alpha _{_{i}}}{\alpha _{j}}+\frac{\beta _{j}}{\beta _{i}}\geq
2-\epsilon ,\ \text{for }i,j=1,....,p.
\end{equation*}

From these last inequalities, and since $0\,<\alpha _{_{1}}\leq ...\leq
\alpha _{_{p}}\leq 1$, $\left\{ \alpha _{_{1}},..,\alpha _{_{p}}\right\}
\subset \left\{ \beta _{_{1}},..,\beta _{_{p}}\right\} $, and using the
above lemma, we obtain $\left\vert \alpha _{_{i}}-\beta _{i}\right\vert \leq
\epsilon $, for $i=1,....,p.$

Since, $P_{_{n}}=\sum_{i=}^{p}\alpha _{_{i}}E_{i}$, and $Q_{_{n}}=%
\sum_{i=}^{p}\beta _{_{i}}E_{i}$, then 
\begin{eqnarray*}
\left\Vert P_{_{n}}-Q_{_{n}}\right\Vert &=&\underset{1\leq i\leq p_{_{n}}}{%
\max }\left\vert \alpha _{_{i}}-\beta _{i}\right\vert \\
&\leq &\epsilon .
\end{eqnarray*}%
Therefore, $P=Q$.

Case 2. $\sigma (Q)\subset \sigma (P).$ Using the same argument as used
before, we find also $P=Q.$

The implication $(ii)\Rightarrow (i)$ follows immediately from $(N-AGMI).$
\end{proof}

\begin{remark}
The above lemma in the original paper $[12]$ is the $Theorem\ 3.6$ but with
equality between spectrum of $P$ and $Q$ instead of the inclusion. The
equality condition is enough for the invertible case. But for the non
invertible case, the lemma presented here with inclusion is needed.
\end{remark}

In the two next propositions, we shall present some characterizations of the
class of all invertible normal operators in $\mathfrak{B}(H)$.

\begin{proposition}
$[15]\ $Let $S$ be an invertible operator in $\mathfrak{B}(H)$. Then, the
following properties are equivalent:

$(i)$ $S$ is normal,

$(ii)$ $\forall X\in \mathfrak{B}(H),\ \left\Vert SXS^{-1}\right\Vert
+\left\Vert S^{-1}XS\right\Vert =\left\Vert S^{\ast }XS^{-1}\right\Vert
+\left\Vert S^{-1}XS^{\ast }\right\Vert ,$

$(iii)$ $\forall X\in \mathfrak{B}(H),\ \left\Vert SXS^{-1}\right\Vert
+\left\Vert S^{-1}XS\right\Vert \geq \left\Vert S^{\ast }XS^{-1}\right\Vert
+\left\Vert S^{-1}XS^{\ast }\right\Vert .$

$(iv)$ $\forall X\in \mathfrak{B}(H),\ \left\Vert SXS^{-1}\right\Vert
+\left\Vert S^{-1}XS\right\Vert \geq 2\left\Vert X\right\Vert .$
\end{proposition}

\begin{proof}
$(i)\Rightarrow (ii).$ This follows from $Remark\ 1$.

$(ii)\Rightarrow (iii).$ This implication is trivial.

$(iii)\Rightarrow (iv).$ This follows immediately from $Corollary\
1.6^{\prime }.$

$(iv)\Rightarrow (i).$ Assume $(iv)$ holds.

Let $S=UP$, $S^{\ast }=U^{\ast }Q$ be the polar decompositions of $S$ and $%
S^{\ast }$, and let $X\in \mathfrak{B}(H)$. Then, from $(iv)$, it follows
that:%
\begin{equation*}
\forall X\in \mathfrak{B}(H),\ \left\Vert PXP^{-1}\right\Vert +\left\Vert
Q^{-1}XQ\right\Vert \geq 2\left\Vert X\right\Vert .
\end{equation*}

Since, $\sigma (P^{2})=\sigma (S^{\ast }S)=\sigma (SS^{\ast })=\sigma
(Q^{2}),$ so from the Spectral Theorem, $\sigma (P)=\sigma (Q).$ Using the
last Lemma, we obtain $P=Q.$ Therefore, $S$ is normal.
\end{proof}

\begin{proposition}
Let $S$ be an invertible operator in $\mathfrak{B}(H)$. The following
properties are equivalent:

$(i)$ $S$ is normal,

$(ii)$ $\forall X\in \mathcal{F}_{_{1}}(H),\ \left\Vert SXS^{-1}\right\Vert
+\left\Vert S^{-1}XS\right\Vert =\left\Vert S^{\ast }XS^{-1}\right\Vert
+\left\Vert S^{-1}XS^{\ast }\right\Vert ,$

$(iii)$ $\forall X\in \mathcal{F}_{_{1}}(H),\ \left\Vert SXS^{-1}\right\Vert
+\left\Vert S^{-1}XS\right\Vert \geq \left\Vert S^{\ast }XS^{-1}\right\Vert
+\left\Vert S^{-1}XS^{\ast }\right\Vert ,$

$(iv)\ \forall X\in \mathcal{F}_{_{1}}(H),\ \left\Vert SXS^{-1}\right\Vert
+\left\Vert S^{-1}XS\right\Vert \geq 2\left\Vert X\right\Vert ,$

$(v)$ $\forall X\in \mathfrak{B}(H),\left\Vert SXS^{-1}\right\Vert
+\left\Vert S^{-1}XS\right\Vert \leq \left\Vert S^{\ast }XS^{-1}\right\Vert
+\left\Vert S^{-1}XS^{\ast }\right\Vert ,$

$(vi)$ $\forall X\in \mathcal{F}_{_{1}}(H),\left\Vert SXS^{-1}\right\Vert
+\left\Vert S^{-1}XS\right\Vert \leq \left\Vert S^{\ast }XS^{-1}\right\Vert
+\left\Vert S^{-1}XS^{\ast }\right\Vert .$
\end{proposition}

\begin{proof}
The implication $(i)\Longrightarrow (ii)$ follows immediately from $Remark\
1 $, the implication $(ii)\Longrightarrow (iii)$ is trivial, and the
implication $(iii)\Longrightarrow (iv)$ follows from $Corollary\ 1.6^{\prime
}$.

$(iv)\Longrightarrow (i)$. Assume $(iv)$ holds.

Put $P=\left\vert S\right\vert $ and $Q=\left\vert S^{\ast }\right\vert $.\
Using $(iv)$ and the polar decomposition of $S$ and $S^{\ast }$, we deduce
the following inequality:%
\begin{equation*}
\forall X\in \mathcal{F}_{_{1}}(H),\ \left\Vert PXP^{-1}\right\Vert
+\left\Vert Q^{-1}XQ\right\Vert \geq 2\left\Vert X\right\Vert .
\end{equation*}

Since $P$ and $Q$ are unitarily equivalent, using $[10,\ Theorem\ 2.1]$, we
find that $P=Q$. This proves $(i)$.

Thus, the four conditions $(i)-(iv)$ are mutually equivalent.

$(i)\Rightarrow (v).$ This implication follows immediately from $Remark\ 1.$

$(v)\Rightarrow (vi).$ This implication is trivial.

$(vi)\Rightarrow (i).$ From $(vi)$, it follows that the following inequality
holds%
\begin{equation*}
\forall x,\ y\in (H)_{_{1}},\ \left\Vert Sx\right\Vert \left\Vert (S^{\ast
})^{-1}y\right\Vert +\left\Vert S^{-1}x\right\Vert \left\Vert S^{\ast
}y\right\Vert \leq \left\Vert S^{\ast }x\right\Vert \left\Vert (S^{\ast
})^{-1}y\right\Vert +\left\Vert S^{-1}x\right\Vert \left\Vert Sy\right\Vert
\end{equation*}

Hence%
\begin{equation*}
\forall x,\ y\in (H)_{_{1}},\ \left( \left\Vert Sx\right\Vert -\left\Vert
S^{\ast }x\right\Vert \right) \left\Vert (S^{\ast })^{-1}y\right\Vert \leq
\left( \left\Vert Sy\right\Vert -\left\Vert S^{\ast }y\right\Vert \right)
\left\Vert S^{-1}x\right\Vert \ \ \ (A)
\end{equation*}

Thus 
\begin{equation*}
\left( \forall x\in (H)_{1},\ \left\Vert Sx\right\Vert \geq \left\Vert
S^{\ast }x\right\Vert \right) \vee \left( \forall x\in (H)_{1},\ \left\Vert
Sx\right\Vert \leq \left\Vert S^{\ast }x\right\Vert \right)
\end{equation*}

Assume that the inequality $\left\Vert Sx\right\Vert \geq \left\Vert S^{\ast
}x\right\Vert $ holds for every $x\in (H)_{_{1}}.$

Since the relation $\frac{1}{\left\Vert T^{-1}\right\Vert }\leq \left\Vert
Tx\right\Vert \leq \left\Vert T\right\Vert $ holds for every invertible
operator $T\in \mathfrak{B}(H)$ and for every $x\in (H)_{_{1}}$, then from $%
(A),$ it follows that:%
\begin{equation*}
\forall x,\ y\in (H)_{_{1}},\ \left\Vert Sx\right\Vert -\left\Vert S^{\ast
}x\right\Vert \leq k\left( \left\Vert Sy\right\Vert -\left\Vert S^{\ast
}y\right\Vert \right)
\end{equation*}%
where $k=\left\Vert S\right\Vert \left\Vert S^{-1}\right\Vert $. So we have:%
\begin{equation*}
\forall x,\ y\in (H)_{_{1}},\ \left\Vert Sx\right\Vert +k\left\Vert S^{\ast
}y\right\Vert \leq \left\Vert S^{\ast }x\right\Vert +k\left\Vert
Sy\right\Vert
\end{equation*}

By taking the infimum over $y\in (H)_{_{1}}$, we obtain:%
\begin{equation*}
\forall x\in H,\ \left\Vert Sx\right\Vert \leq \left\Vert S^{\ast
}x\right\Vert .
\end{equation*}

Therefore $S$ is normal.

With the second assumption and by the same argument, we find also that $S$
is normal. This proves $(i)$.
\end{proof}

\begin{remark}
(1) The equivalences between $(i),\ (v)$, and $(vi)$ was given in $[17].$

(2) For $\mathcal{R\in }\left\{ \leq ,\ \geq ,\ =\right\} $, and $\mathfrak{L%
}(H)\in \left\{ \mathfrak{B}(H),\ \mathcal{F}_{_{1}}(H)\right\} $, then the
class of all invertible normal operators in $\mathfrak{B}(H)$ is exactly the
class of all invertible operators $S\in \mathfrak{B}(H)$ satisfying each of
the two following operator inequalities:%
\begin{eqnarray*}
\forall X &\in &\mathfrak{L}(H),\ \left\Vert SXS^{-1}\right\Vert +\left\Vert
S^{-1}XS\right\Vert \ \mathcal{\geq \ }2\left\Vert X\right\Vert , \\
\forall X &\in &\mathfrak{L}(H),\ \left\Vert SXS^{-1}\right\Vert +\left\Vert
S^{-1}XS\right\Vert \ \mathcal{R\ }\left\Vert S^{\ast }XS^{-1}\right\Vert
+\left\Vert S^{-1}XS^{\ast }\right\Vert .
\end{eqnarray*}
\end{remark}

In the next proposition, we shall give a complete characterization of the
class of all normal operators in $\mathfrak{B}(H)$ in terms of operator
inequality. To prove this, we need the following results of Halmos (see $[7]$%
) that says, the set 
\begin{equation*}
\mathfrak{D}(H)=\left\{ S\in \mathfrak{B}(H):S\text{ is left invertible or
right invertible}\right\}
\end{equation*}%
is dense in $\mathfrak{B}(H)$, and from the fact that for $S\in \mathfrak{B}%
(H),$ we have:

$(i)\ S$ is left invertible if and only if $S$ is injective with closed
range,

$(ii)$ $S$ is right invertible if and only if $S$ is surjective.

\begin{proposition}
$[19,\ 20]\ $Let $S\in \mathfrak{B}(H)$. Then, the following properties are
equivalent

$(i)$ $S$ is normal,

$(ii)$ $\forall X\in \mathfrak{B}(H),\ \left\Vert S^{2}X\right\Vert
+\left\Vert XS^{2}\right\Vert \geq 2\left\Vert SXS\right\Vert .$
\end{proposition}

\begin{proof}
We may assume that $S\neq 0$

$(i)\Rightarrow (ii)$. Assume $(i)$ holds.

Let $X\in \mathfrak{B}(H)$.\ Then we have 
\begin{eqnarray*}
\left\Vert S^{2}X\right\Vert +\left\Vert XS^{2}\right\Vert &=&\left\Vert
S^{\ast }SX\right\Vert +\left\Vert XSS^{\ast }\right\Vert ,\ \text{(from }%
Remark\ 1\text{),} \\
&\geq &2\left\Vert SXS\right\Vert ,\text{ \ (from }(N-AGMI)\text{).}
\end{eqnarray*}

This proves $(ii)$.

$(ii)\Rightarrow (i)$. Assume $(ii)$ holds. \ 

We prove $(i)$ in three cases: $S$ injective with closed range, $S$
surjective, and $S$ arbitrary.

$Case\ 1$. Assume that $S$ is injective with closed range.

Hence, $S^{+}S=I$, $\ker P=\ker S=\left\{ 0\right\} $, and $R(P)=R(S^{\ast
}S)$ is closed (since $R(S^{\ast })$ is also closed). Thus $\ker P=\left\{
0\right\} $ and $R(P)=\left( \ker P\right) ^{\bot }=H.$ So, $P$ is
invertible.

All the $2\times 2$ matrices used in this proof are given with respect to
the orthogonal direct sum $H=R(S)\oplus \ker S^{\ast }.$ Then $S=\left[ 
\begin{array}{cc}
S_{_{1}} & S_{_{2}} \\ 
0 & 0%
\end{array}%
\right] \,$. We put $P=\left\vert S\right\vert ,\ Q=\left\vert S^{\ast
}\right\vert ,\ P_{_{1}}=\left\vert S_{_{1}}\right\vert ,\
P_{_{2}}=\left\vert S_{_{2}}\right\vert ,\ Q_{_{1}}=\left(
S_{_{1}}S_{_{1}}^{\ast }+S_{_{2}}S_{_{2}}^{\ast }\right) ^{\frac{1}{2}}.$ So
we have $S^{\ast }S=P^{2}=\left[ 
\begin{array}{cc}
P_{_{1}}^{2} & S_{_{1}}^{\ast }S_{_{2}} \\ 
S_{_{2}}^{\ast }S_{_{1}} & P_{_{2}}^{2}%
\end{array}%
\right] ,\ SS^{\ast }=Q^{2}=\left[ 
\begin{array}{cc}
Q_{_{1}}^{2} & 0 \\ 
0 & 0%
\end{array}%
\right] .$ It is clear that $Q_{1}$ is invertible and $Q^{+}=\left[ 
\begin{array}{cc}
Q_{_{1}}^{-1} & 0 \\ 
0 & 0%
\end{array}%
\right] .$

From $(ii),$ the two following inequalities hold%
\begin{equation}
\forall X\in \mathfrak{B}(H),\ \left\Vert S^{2}S^{+}XS^{+}\right\Vert
+\left\Vert S^{+}XS\right\Vert \geq 2\left\Vert SS^{+}X\right\Vert . 
\tag{$1$}
\end{equation}%
\begin{equation}
\forall X\in \mathfrak{B}(H),\ \left\Vert XS\right\Vert +\left\Vert
S^{2}XS^{+}\right\Vert \geq 2\left\Vert SX\right\Vert .  \tag{$2$}
\end{equation}

The proof is given in four steps.

Step 1. Prove that $\left( S^{2}\right) ^{+}S=S^{+}.$

It is known that $S^{+}$ is the unique solution of the following four
equations: $SXS=S,\ XSX=X,\ (XS)^{\ast }=XS,\ (SX)^{\ast }=SX.$ It is easy
to see that $\left( S^{2}\right) ^{+}S$ satisfies the first three equations.

Now we prove that $\left( S^{2}\right) ^{+}S$ also satisfies the last
equation. Since the operator $S\left( S^{2}\right) ^{+}S$ is a projection,
it suffices to prove that its norm is less than or equal to one. By taking $%
X=\left( S^{2}\right) ^{+}S$ in $(2)$, we obtain%
\begin{equation*}
2\geq \left\Vert \left( S^{2}\right) ^{+}S^{2}\right\Vert +\left\Vert
S^{2}\left( S^{2}\right) ^{+}SS^{+}\right\Vert \geq 2\left\Vert S\left(
S^{2}\right) ^{+}S\right\Vert .
\end{equation*}

Hence $\left\Vert S\left( S^{2}\right) ^{+}S\right\Vert \leq 1.$ Therefore $%
\left( S^{2}\right) ^{+}S=S^{+}.$

Step 2. Prove that $\left( S^{2}\right) ^{+}=(S^{+})^{2}$.

Since $S^{2}\left( S^{2}\right) ^{+}=SS^{+}S^{2}\left( S^{2}\right) ^{+}$,
then $S^{2}\left( S^{2}\right) ^{+}=S^{2}\left( S^{2}\right) ^{+}SS^{+}.$ So
from step 2, we obtain $S^{2}\left( S^{2}\right) ^{+}=S^{2}(S^{+})^{2}.$
Since $S^{2}$ is injective, we have $\left( S^{2}\right) ^{+}=(S^{+})^{2}.$

Step 3. Prove that $\ker S^{\ast }=\left\{ 0\right\} .$

Since $S$ is injective, then $\ker S^{\ast }=\left\{ 0\right\} $ if and only
if $S_{_{2}}=0.$ Assume that $S_{_{2}}\neq 0.$

Since $\left( S^{2}\right) ^{+}=(S^{+})^{2}$, then the two operators $%
S^{\ast }S$ and $SS^{+}$ commute (see $[2],\ [9]$ ). Thus $P^{2}=\left[ 
\begin{array}{cc}
P_{_{1}}^{2} & 0 \\ 
0 & P_{2}^{2}%
\end{array}%
\right] .$ So that $P=\left[ 
\begin{array}{cc}
P_{_{1}} & 0 \\ 
0 & P_{2}%
\end{array}%
\right] .$

Since $\ker S^{\ast }\neq \left\{ 0\right\} $, then $\sigma (Q^{2})=\sigma
(Q_{1}^{2})\cup \left\{ 0\right\} $. From the fact that $\sigma
(P^{2})=\sigma (Q^{2})-\left\{ 0\right\} $, we have $\sigma (P^{2})=\sigma
(Q_{1}^{2})$. Then $\sigma (P_{_{1}}^{2})\cup \sigma (P_{2}^{2})=\sigma
(Q_{1}^{2})$. Hence $\sigma (P_{_{1}}^{2})\subset \sigma (Q_{1}^{2})$. Thus $%
\sigma (P_{_{1}})\subset \sigma (Q_{_{1}})$.

Using the polar decomposition of $S$ and $S^{\ast }$ in the inequality $(1)$%
, we obtain the following inequality

\begin{equation*}
\forall X\in \mathfrak{B}(H),\ \left\Vert S^{2}S^{+}XP^{-1}\right\Vert
+\left\Vert Q^{+}XQ\right\Vert \geq 2\left\Vert SS^{+}X\right\Vert
\end{equation*}

By taking $X=\left[ 
\begin{array}{cc}
X_{_{1}} & 0 \\ 
0 & 0%
\end{array}%
\right] $ (resp. $X=\left[ 
\begin{array}{cc}
0 & X_{2} \\ 
0 & 0%
\end{array}%
\right] $ ), where $X_{_{1}}\in \mathfrak{B}(R(S))$ (resp. $X_{2}\in 
\mathfrak{B}(\ker S^{\ast },R(S)))$ in the last inequality and since $%
S^{2}S^{+}=\left[ 
\begin{array}{cc}
S_{_{1}} & 0 \\ 
0 & 0%
\end{array}%
\right] $, we deduce the two following inequalities%
\begin{equation}
\forall X_{_{1}}\in \mathfrak{B}(R(S)),\ \left\Vert
P_{1}X_{_{1}}P_{1}^{-1}\right\Vert +\left\Vert
Q_{_{1}}^{-1}X_{_{1}}Q_{_{1}}\right\Vert \geq 2\left\Vert X_{1}\right\Vert 
\tag{$3$}
\end{equation}

\begin{equation}
\forall X_{2}\in \mathfrak{B}(\ker S^{\ast },R(S)),\ \left\Vert
P_{1}X_{2}P_{2}^{-1}\right\Vert \geq 2\left\Vert X_{2}\right\Vert  \tag{$4$}
\end{equation}

By taking $X_{2}=x\otimes y$ (where $x\in (R(S))_{1}$, $y\in \ker S^{\ast }$%
) in $(4)$, we obtain 
\begin{equation*}
\forall x\in (R(S))_{1},\forall y\in \ker S^{\ast },\ \left\Vert
P_{1}x\right\Vert \left\Vert P_{2}^{-1}y\right\Vert \geq 2\left\Vert
y\right\Vert
\end{equation*}

So we have%
\begin{equation*}
\forall x\in (R(S))_{1},\forall y\in \left( \ker S^{\ast }\right) _{1},\
\left\Vert P_{1}x\right\Vert \geq 2\left\Vert P_{2}y\right\Vert
\end{equation*}

Thus $\left\Vert P_{2}y\right\Vert \leq \frac{k}{2}$, for every $y\in \left(
\ker S^{\ast }\right) _{1}$ (where $k=\underset{\left\Vert x\right\Vert =1}{%
\inf }\left\Vert P_{1}x\right\Vert >0$). Then $\left\langle
P_{2}^{2}y,y\right\rangle \leq \frac{k^{2}}{4}$, for every $y\in \left( \ker
S^{\ast }\right) _{1}$. So we obtain $\sigma (P_{2}^{2})\subset (0,\frac{%
k^{2}}{4}]$ and $\sigma (P_{1}^{2})\subset \lbrack k^{2},\infty ).$

Since $\sigma (P_{_{1}})\subset \sigma (Q_{_{1}})$, and $P_{_{1}}$ ,\ $%
Q_{_{1}}$ satisfy the inequality $\left( 3\right) $, then using $Lemma\ 5$,
we obtain $P_{1}=Q_{_{1}}$. Hence $\sigma (Q_{_{1}}^{2})=\sigma
(P_{1}^{2})=\sigma (P_{_{1}}^{2})\cup \sigma (P_{2}^{2})$. Then $\sigma
(P_{2}^{2})\subset $ $\sigma (P_{1}^{2})$, that is impossible since $(0,%
\frac{k^{2}}{4}]\cap \lbrack k^{2},\infty )=\varnothing $. Therefore $\ker
S^{\ast }=\left\{ 0\right\} $.

Step 4. Prove that $S$ is normal.

Since $\ker S^{\ast }=\left\{ 0\right\} $, then $R(S)=H$. So that $S$ is
invertible and satisfies the inequality $(ii)$. Hence $S$ satisfies the
following inequality%
\begin{equation*}
\forall X\in \mathfrak{B}(H),\ \left\Vert SXS^{-1}\right\Vert +\left\Vert
S^{-1}XS\right\Vert \geq 2\left\Vert X\right\Vert
\end{equation*}

Therefore $S$ is normal, by using $Proposition\ 2$.

$Case$ $2$. Assume $S$ surjective.

Then $S^{\ast }$ is injective with a closed range satisfying also the
inequality $(ii)$. So that from case 1, $S^{\ast }$ is normal. Hence $S$ is
normal.

$Case$ $3$. General situation.

We may assume without loss of generality that $\left\Vert S\right\Vert =1.$%
Then $\left\Vert S^{2}\right\Vert =\left\Vert S\right\Vert ^{2}=1.$ There
exists a sequence $\left( S_{n}\right) _{n\geq 1}$ of elements in $\mathfrak{%
D}(H)$ such that $S_{n}\rightarrow S$ uniformly.

Define the real function $F$ on the complete metric space $\left( \mathfrak{B%
}(H)\right) _{1}$ by 
\begin{equation*}
\forall X\in \left( \mathfrak{B}(H)\right) _{1},\ F(X)=\left\Vert
S^{2}X\right\Vert +\left\Vert XS^{2}\right\Vert -2\left\Vert SXS\right\Vert ,
\end{equation*}%
and for $n\geq 1,$ define the real function $F_{n}$ on $\left( \mathfrak{B}%
(H)\right) _{1}$ by 
\begin{equation*}
\forall X\in \left( \mathfrak{B}(H)\right) _{1},\ F_{n}(X)=\left\Vert
S_{n}^{2}X\right\Vert +\left\Vert XS_{n}^{2}\right\Vert -2\left\Vert
S_{n}XS_{n}\right\Vert .
\end{equation*}

Put $D=\left\{ X\in \left( \mathfrak{B}(H)\right) _{1}:F(X)>0\right\} .$
Then there are two cases, $D=\varnothing $, $D\neq \varnothing .$

$\mathbf{(1).}\ D=\varnothing .$ So, it follows that%
\begin{equation}
\forall X\in \mathfrak{B}(H),\left\Vert S^{2}X\right\Vert +\left\Vert
XS^{2}\right\Vert =2\left\Vert SXS\right\Vert .  \tag{$\ast $}
\end{equation}

From this equality, we have%
\begin{equation*}
\forall x,y\in H,\ \left\Vert S^{2}x\right\Vert \left\Vert y\right\Vert
+\left\Vert x\right\Vert \left\Vert S^{\ast 2}y\right\Vert =2\left\Vert
Sx\right\Vert \left\Vert S^{\ast }y\right\Vert .
\end{equation*}

Since $S^{2}\neq 0$, and from this last inequality, we deduce easily that $S$
and $S^{\ast }$ are injective, and then $S$ is with dense range.

Prove now that $S$ is with closed range. Let $(x_{_{n}})$ be a sequence of
vectors in $H$ such that $(Sx_{_{n}})$ converges to a vector $y\in H.$ We
may choose a vector $u\in (H)_{_{1}}$ such that $S^{\ast 2}u\neq 0.$ From
the above inequality, we obtain%
\begin{equation*}
\forall n,m\geq 1,\ \left\Vert S^{2}x_{_{n}}-S^{2}x_{_{m}}\right\Vert
+\left\Vert x_{_{n}}-x_{_{m}}\right\Vert \left\Vert S^{\ast 2}u\right\Vert
=2\left\Vert Sx_{_{n}}-Sx_{_{m}}\right\Vert \left\Vert S^{\ast }u\right\Vert
.
\end{equation*}

Hence, $(x_{_{n}})$ is a Cauchy sequence, and then it converges to some
vector $x\in H$. So that $Sx_{_{n}}\rightarrow y=Sx.$ This proves $R(S)$
closed. Then, $S$ is invertible.

So, from $(\ast )$, it follows that 
\begin{equation*}
\forall X\in \mathfrak{B}(H),\ \left\Vert SXS^{-1}\right\Vert +\left\Vert
S^{-1}XS\right\Vert =2\left\Vert X\right\Vert .
\end{equation*}

From $Proposition\ 2$, $(i)$ holds.

$\mathbf{(2).}$ $D\neq \varnothing .$ From the fact that $F$ is a positive
continuous map on $\left( \mathfrak{B}(H)\right) _{1}$, it follows that 
\begin{equation*}
\overline{D}=\overline{F^{-1}\left( \left( 0,\infty \right) \right) }%
=F^{-1}\left( [0,\infty )\right) =\left\{ X\in \left( \mathfrak{B}(H)\right)
_{1}:F(X)\geq 0\right\} =\left( \mathfrak{B}(H)\right) _{1}.
\end{equation*}

Let $X\in D$, and $\epsilon >0.$ Since $S_{n}\rightarrow S$ uniformly, then
there exists an integer $N\geq 1$ (depends only in $\epsilon $) such that 
\begin{equation*}
\forall n\geq N,\ \forall Y\in \left( \mathfrak{B}(H)\right) _{1},\
\left\vert F(Y)-F_{n}(Y)\right\vert \leq \epsilon .
\end{equation*}

If there exists $n\geq N$ such that $F_{n}(X)<0$, then using this last
inequality, we have $0\leq F(X)<\epsilon $, for every $\epsilon >0$; thus $%
F(X)=0$, leading a contradiction with $X\in D.$

From this fact, it follows that%
\begin{equation*}
\forall X\in D,\ \forall n\geq N,\ \ F_{n}(X)\geq 0.
\end{equation*}

Since each $F_{n}$ is a continuous map on $\left( \mathfrak{B}(H)\right)
_{1} $ and $L$ is dense in $\left( \mathfrak{B}(H)\right) _{1}$, then%
\begin{equation*}
\forall X\in \left( \mathfrak{B}(H)\right) _{1},\ \forall n\geq N,\
F_{n}(X)\geq 0.
\end{equation*}

So, it follows that%
\begin{equation*}
\forall X\in \mathfrak{B}(H),\ \forall n\geq N,\ \left\Vert
S_{n}^{2}X\right\Vert +\left\Vert XS_{n}^{2}\right\Vert \geq 2\left\Vert
S_{n}XS_{n}\right\Vert .
\end{equation*}%
Since for each $n\geq 1$, $S_{n}$ $\in \mathfrak{D}(H)$, then from the two
above cases, we obtain that $S_{n}$ is a normal operator, for every $n\geq N$%
. Since $S_{n}\rightarrow S$ uniformly and the class of all normal operators
in $\mathfrak{B}(H)$ is closed, then $S$ is a normal.
\end{proof}

\begin{remark}
In the above proof, the case 1 and the case 2 are given in the lemma
presented in the corrigendum $[19]$, and the general situation is presented
in $[20]$. Note that in the proof of this lemma of the corrigendum, we have
used the $Theorem$ $3.6$ of $[12]$ which is given with the condition of
equality between spectrum (that is not suffice), but the lemma is true with
the condition of inclusion instead of equality. So, we have mentioned in the
proof of the lemma that this theorem remains true with inclusion between
spectrum, but without argument. In this survey, we have present this
argument (see $Lemma$ $4$ and $Lemma$ $5$).
\end{remark}

\begin{corollary}
Let $S\in \mathfrak{B}(H)$. Then, the following properties are equivalent.

$(i)$ $S$ is normal,

$(ii)$ $\forall X\in \mathfrak{B}(H),\ \left\Vert S^{2}X\right\Vert
\left\Vert XS^{2}\right\Vert \geq \left\Vert SXS\right\Vert ^{2}.$
\end{corollary}

\begin{proof}
$(i)\Rightarrow (ii)$. Assume $(i)$ holds, and let $X\in \mathfrak{B}(H)$.\
Then we have%
\begin{eqnarray*}
\left\Vert S^{2}X\right\Vert \left\Vert XS^{2}\right\Vert &=&\left\Vert
S^{\ast }SX\right\Vert \left\Vert XSS^{\ast }\right\Vert \text{,\ \ (using }%
Remark\text{ }1\text{)} \\
&\geq &\left\Vert SXS\right\Vert ^{2}\text{, (see the proof of the }(N-AGMI)%
\text{ in }Proposition\text{ }1\text{).}
\end{eqnarray*}

$(ii)\Rightarrow (i)$. Assume $(ii)$ holds, and let $X\in \mathfrak{B}(H)$.\
Then we have:%
\begin{eqnarray*}
\frac{\left\Vert S^{2}X\right\Vert +\left\Vert XS^{2}\right\Vert }{2} &\geq &%
\sqrt{\left\Vert S^{2}X\right\Vert \left\Vert XS^{2}\right\Vert }\text{, \ \
(from the numerical AGMI),} \\
&\geq &\left\Vert SXS\right\Vert \text{, \ (using }(ii)\text{).}
\end{eqnarray*}

From the last proposition, $(i)$ holds.
\end{proof}

\begin{corollary}
$[18]\ $Let $S$ be an operator with closed range in $\mathfrak{B}(H)$. Then
the following properties are equivalent:%
\begin{equation*}
\begin{array}{ccccc}
(i) & S\text{ is normal,} &  &  &  \\ 
(ii) & \forall X\in \mathfrak{B}(H), & \left\Vert SXS^{+}\right\Vert
+\left\Vert S^{+}XS\right\Vert & = & \left\Vert S^{\ast }XS^{+}\right\Vert
+\left\Vert S^{+}XS^{\ast }\right\Vert , \\ 
(iii) & \forall X\in \mathfrak{B}(H), & \left\Vert SXS^{+}\right\Vert
+\left\Vert S^{+}XS\right\Vert & \geq & \left\Vert S^{\ast
}XS^{+}\right\Vert +\left\Vert S^{+}XS^{\ast }\right\Vert , \\ 
(iv) & \forall X\in \mathfrak{B}(H), & \left\Vert SXS^{+}\right\Vert
+\left\Vert S^{+}XS\right\Vert & \geq & 2\left\Vert SS^{+}XS^{+}S\right\Vert
,%
\end{array}%
\end{equation*}
\end{corollary}

\begin{proof}
We may assume that $S\neq 0.$

$(i)\Rightarrow (ii).$ This follows immediately from $Remark\ 1$.

The implication $(ii)\Rightarrow (iii)$ is trivial.

$(iii)\Rightarrow (vi).$ This follows immediately from $Corollary\
1.5^{\prime }$.

$(iv)\Rightarrow (i)$. Assume $(iv)$ holds. Then the following inequality
holds%
\begin{equation*}
\forall X\in \mathfrak{B}(H),\ \left\Vert S^{2}XSS^{+}\right\Vert
+\left\Vert S^{+}SXS^{2}\right\Vert \geq 2\left\Vert
SS^{+}SXSS^{+}S\right\Vert .
\end{equation*}

From this inequality and since $\left\Vert SS^{+}\right\Vert =$ $\left\Vert
S^{+}S\right\Vert =1,\ $and $SS^{+}S=S$, it follows that%
\begin{equation*}
\forall X\in \mathfrak{B}(H),\ \left\Vert S^{2}X\right\Vert +\left\Vert
XS^{2}\right\Vert \geq 2\left\Vert SXS\right\Vert .
\end{equation*}

Using $Proposition\ 4$, $S$ is normal.
\end{proof}

\begin{remark}
In the original paper $[18]$, the above corollary is presented before the
characterization of \ the class of all normal operators in $\mathfrak{B}(H)$
in its general situation, for this reason its proof was very strong. But, in
this survey, we have adopt a new strategy, where this corollary follows
immediately from the general situation.
\end{remark}

In $[1,\ 1972]\,$, Ando proved that for $S\in \mathfrak{B}(H)$, $S$ is
normal if and only if $S$ and $S^{\ast }$ are paranormal, and $\ker S=\ker
S^{\ast }$. In the next proposition, we present some new general
characterizations of the class of all normal operators in $\mathfrak{B}(H)$,
and we shall show that the Ando result remains true without the kernel
assumption.

\begin{proposition}
Let $S\in \mathfrak{B}(H)$. The following properties are equivalent.

$(i)$ $S$ is normal,

$(ii)$ $\left\vert S^{2}\right\vert =\left\vert S\right\vert ^{2},\
\left\vert S^{\ast 2}\right\vert =\left\vert S^{\ast }\right\vert ^{2},$

$(iii)\ \left\vert S^{2}\right\vert ^{2}\geq \left\vert S\right\vert ^{4},\
\left\vert S^{\ast 2}\right\vert ^{2}\geq \left\vert S^{\ast }\right\vert
^{4},$

$(iv)\ S$ and $S^{\ast }$ belong to class $\mathbf{A}$,

$(v)\ S$ and $S^{\ast }$ are paranormal.
\end{proposition}

\begin{proof}
The two implications $(i)\Rightarrow (ii),\ (ii)\Rightarrow (iii)$ are
trivial.

$(iii)\Rightarrow (iv)$.\ This follows from L\u{o}wner-Heinz inequality with 
$\alpha =\frac{1}{2}$.

$(vi)\Rightarrow (v)$. This follows from $[8]$.

$(v)\Rightarrow (i).$ Assume $(v)$ holds. Then, we have%
\begin{equation*}
\left\{ 
\begin{array}{c}
(1)\ \ \ \forall x\in H,\ \left\Vert x\right\Vert \left\Vert
S^{2}s\right\Vert \geq \left\Vert Sx\right\Vert ^{2}\ , \\ 
(2)\ \ \ \forall x\in H,\ \left\Vert x\right\Vert \left\Vert \left( S^{\ast
}\right) ^{2}x\right\Vert \geq \left\Vert S^{\ast }x\right\Vert ^{2}\ .%
\end{array}%
\right.
\end{equation*}

So, it follows that%
\begin{equation*}
(3)\ \ \ \forall X\in \mathcal{F}_{_{1}}(H),\ \left\Vert S^{2}X\right\Vert
+\left\Vert XS^{2}\right\Vert \geq 2\left\Vert SXS\right\Vert \ .
\end{equation*}

The proof is given in three cases.

Case 1.\ Assume that $S$ surjective.

From $(1)$, it follows that 
\begin{equation*}
\forall x\in H,\ \left\Vert S^{+}x\right\Vert \left\Vert Sx\right\Vert \geq
\left\Vert x\right\Vert ^{2}.
\end{equation*}

Then $S$ is injective. Hence $S$ is invertible. So using $(3)$, we obtain%
\begin{equation*}
\forall X\in \mathcal{F}_{_{1}}(H),\ \left\Vert SXS^{-1}\right\Vert
+\left\Vert S^{-1}XS\right\Vert \geq 2\left\Vert X\right\Vert .
\end{equation*}

So from $Proposition\ 3$, $S$ is normal.

Case 2. Assume that $S$ is injective with closed range.

Then $S^{\ast }$ is surjective. So, from $(2)$ and using the same argument
as used in the case 1, we find that $S$ invertible. So that $S$ is normal.

Case 3. General situation.

From $(3)$, and by using the same argument as used in the case 3 of the
proof of $Proposition$ $4$ (where $\mathcal{F}_{_{1}}(H)$ takes the place of 
$\mathfrak{B}(H)$), and using $[10,\ Theorem\ 2.1],$ we obtain that $S$ is
normal.
\end{proof}

\begin{remark}
The equivalences between $(i),\ (ii),\ $and $(iii)$ in the above corollary
was given in $[20],$ and follow from $Proposition\ 4$. The equivalences
between $(i),\ (iv)$, and $(v)$ are new.
\end{remark}

\section{Arithmetic-Geometric-Mean Inequality, Selfadjoint Operators, and
Characterization}

In the following proposition, we shall give a family of operator
inequalities that are equivalent to \bigskip $(S-AGMI)$ and presenting the
proof of $\left( S1\right) $ given in $[5].$

\begin{proposition}
$[3,\ 5]\ $The following operator inequalities hold and are mutually
equivalent:%
\begin{equation}
\forall X\in \mathfrak{B}(H),\ \left\Vert A^{\ast }AX+XBB^{\ast }\right\Vert
\geq 2\left\Vert AXB\right\Vert ,  \tag{$1$}
\end{equation}%
for every $A,B\in \mathfrak{B}(H),$%
\begin{equation}
\forall X\in \mathfrak{B}(H),\ \left\Vert SXR^{+}+S^{+}XR\right\Vert \geq
2\left\Vert SS^{+}XR^{+}R\right\Vert ,  \tag{$2$}
\end{equation}%
for every selfadjoint operators with closed ranges $S,R\in \mathfrak{B}(H),$%
\begin{equation}
\forall X\in \mathfrak{B}(H),\ \left\Vert SXR^{-1}+S^{-1}XR\right\Vert \geq
2\left\Vert X\right\Vert ,  \tag{$3$}
\end{equation}%
for every invertible selfadjoint operators $S,R\in \mathfrak{B}(H),$%
\begin{equation}
\forall X\in \mathfrak{B}(H),\ \left\Vert S^{2}X+XR^{2}\right\Vert \geq
2\left\Vert SXR\right\Vert ,  \tag{$4$}
\end{equation}%
for every selfadjoint operators $S,R\in \mathfrak{B}(H),$ 
\begin{equation}
X\in \mathfrak{B}(H),\ \left\Vert A^{\ast }AX+XAA^{\ast }\right\Vert \geq
2\left\Vert AXA\right\Vert ,  \tag{$1^{\prime }$}
\end{equation}%
for every $A\in \mathfrak{B}(H),$ 
\begin{equation}
\forall X\in \mathfrak{B}(H),\ \left\Vert SXS^{+}+S^{+}XS\right\Vert \geq
2\left\Vert SS^{+}XS^{+}S\right\Vert ,  \tag{$2^{\prime }$}
\end{equation}%
for every selfadjoint operator with closed range $S\in \mathfrak{B}(H),$%
\begin{equation}
\forall X\in \mathfrak{B}(H),\ \left\Vert SXS^{-1}+S^{-1}XS\right\Vert \geq
2\left\Vert X\right\Vert ,  \tag{$3^{\prime }$}
\end{equation}%
for every invertible selfadjoint operator $S\in \mathfrak{B}(H),$%
\begin{equation}
\forall X\in \mathfrak{B}(H),\ \left\Vert S^{2}X+XS^{2}\right\Vert \geq
2\left\Vert SXS\right\Vert ,  \tag{$4^{\prime }$}
\end{equation}%
for every selfadjoint operator $S\in \mathfrak{B}(H).$
\end{proposition}

\begin{proof}
$(1)\Rightarrow (2).$ Assume $(1)$ holds. Let $S,R\in \mathbb{S}_{cr}(H),\
X\in \mathfrak{B}(H)$. Since $S=S^{\ast }SS^{+}$ and $R=R^{+}RR^{\ast }$,
then from $(1)$ it follows that 
\begin{eqnarray*}
\left\Vert SXR^{+}+S^{+}XR\right\Vert &=&\left\Vert S^{\ast }S\left(
S^{+}XR^{+}\right) +\left( S^{+}XR^{+}\right) RR^{\ast }\right\Vert \\
&\geq &2\left\Vert SS^{+}XR^{+}R\right\Vert .
\end{eqnarray*}

Hence $(2)$ holds.

$(2)\Rightarrow (3).$ Trivial.

$(3)\Rightarrow (4).$ Assume $(3)$ holds. Let $S,R\in \mathbb{S}(H)$, and
put $P=\left\vert S\right\vert $, $Q=\left\vert R\right\vert $.

Let $\epsilon >0$. From (4), the following inequality holds%
\begin{equation*}
\forall X\in \mathfrak{B}(H),\ \left\Vert \left( P+\epsilon I\right)
^{2}X+X\left( Q+\epsilon I\right) ^{2}\right\Vert \geq 2\left\Vert \left(
P+\epsilon I\right) X\left( Q+\epsilon I\right) \right\Vert .
\end{equation*}

Letting $\epsilon \rightarrow \infty $, we obtain%
\begin{equation*}
\forall X\in \mathfrak{B}(H),\ \left\Vert S^{2}X+XR^{2}\right\Vert \geq
2\left\Vert SXR\right\Vert .
\end{equation*}

This proves $(4)$.

$(4)\Rightarrow (1).$ This follows immediately by using the polar
decomposition of an operator.

Hence, the equivalences $(1)-(4)$ hold.

From pair of operators to single operator, the equivalences $(1^{\prime
})-(5^{\prime })$ hold.

$(1)\Rightarrow (1^{\prime }).$ Trivial.

$(1^{\prime })\Rightarrow (1)$. This follows using Berberian technic as used
in $Proposition\ 1$.

Hence, the eight properties are mutually equivalent.

Prove now that the operator inequality $(3^{\prime })$ holds.

Step1.\ Let $S,\ X\in \mathfrak{B}(H)$ such that $S$ and $X$ are
selfadjoint, and $S$ invertible.

Then, there exists $\lambda \in \sigma (X)$ such that $\left\vert \lambda
\right\vert =\left\Vert X\right\Vert .$ Since, $\sigma (X)=\sigma
(SXS^{-1})\subset V(SXS^{-1})$, there exists a state $f$ on $\mathfrak{B}(H)$
such that $\lambda =f\left( SXS^{-1}\right) =f\left( S^{-1}XS\right) .$ This
gives us, $2\left\Vert X\right\Vert =\left\vert f\left(
SXS^{-1}+S^{-1}XS\right) \right\vert \leq \left\Vert
SXS^{-1}+S^{-1}XS\right\Vert .$

Step 2. Let $S,\ X\in \mathfrak{B}(H)$ such that $S$ is selfadjoint
invertible.

Let the two following operators on the Hilbert space $H\oplus H$ given by $M=%
\left[ 
\begin{array}{cc}
S & 0 \\ 
0 & S%
\end{array}%
\right] $ and $Y=\left[ 
\begin{array}{cc}
0 & X \\ 
X^{\ast } & 0%
\end{array}%
\right] .$ So that $M$ and $Y$ are selfadjoint operators in $\mathfrak{B}%
(H\oplus H)$ and where $M$ is invertible. Applying step 1 for this pair of
operators, so we obtain 
\begin{eqnarray*}
\left\Vert SXS^{-1}+S^{-1}XS\right\Vert &=&\left\Vert
MXM^{-1}+M^{-1}XM\right\Vert , \\
&\geq &2\left\Vert Y\right\Vert , \\
&=&2\left\Vert X\right\Vert .
\end{eqnarray*}
\end{proof}

\begin{corollary}
The following operator inequalities hold and are equivalent to $(S-AGMI):$

\begin{equation}
\forall X\in \mathfrak{B}(H),\ \left\Vert S^{\ast }XR^{+}+S^{+}XR^{\ast
}\right\Vert \geq 2\left\Vert SS^{+}XR^{+}R\right\Vert ,  \tag{$5$}
\end{equation}%
for every operators with closed ranges $S,\ R\in \mathfrak{B}(H),$

\begin{equation}
\forall X\in \mathfrak{B}(H),\ \left\Vert S^{\ast }XR^{-1}+S^{-1}XR^{\ast
}\right\Vert \geq 2\left\Vert X\right\Vert ,  \tag{$6$}
\end{equation}%
for every invertible operators $S,\ R\in \mathfrak{B}(H),$

\begin{equation}
\forall X\in \mathfrak{B}(H),\ \left\Vert S^{\ast }XS^{+}+S^{+}XS^{\ast
}\right\Vert \geq 2\left\Vert SS^{+}XS^{+}S\right\Vert ,  \tag{$5^{\prime }$}
\end{equation}%
for every operator with closed range $S\in \mathfrak{B}(H),$

\begin{equation}
\forall X\in \mathfrak{B}(H),\ \left\Vert S^{\ast }XS^{-1}+S^{-1}XS^{\ast
}\right\Vert \geq 2\left\Vert X\right\Vert ,  \tag{$6^{\prime }$}
\end{equation}%
for every invertible operator $S\in \mathfrak{B}(H).$
\end{corollary}

\begin{proof}
Assume $(S-AGMI)$ holds. Prove that $(5)$ holds.

Let $S,R\in \mathcal{R}(H)$, and $X\in \mathfrak{B}(H)$. Since, $SS^{+}S=S$
and $RR^{+}R=R$, then we have 
\begin{eqnarray*}
\left\Vert S^{\ast }XR^{+}+S^{+}XR^{\ast }\right\Vert &=&\left\Vert S^{\ast
}S\left( S^{+}XR^{+}\right) +\left( S^{+}XR^{+}\right) RR^{\ast }\right\Vert
, \\
&\geq &2\left\Vert SS^{+}XR^{+}R\right\Vert \text{, \ \ \ (from }(S-AGMI)%
\text{).}
\end{eqnarray*}

This proves $(5)$.

It is clear that $(5)$ implies $(6),\ (5^{\prime }),\ (6^{\prime })$, and
using $Remark\ 1$, it is clear that $(5)$ (resp. $(6),\ (5^{\prime }),\
(6^{\prime })$) implies $(2)$ (resp. $(3),\ (2^{\prime }),\ (3^{\prime })$).
\end{proof}

Note that the six operator inequalities $(2)-(4)$ and $(2^{^{\prime
}})-(4^{^{\prime }})$ given in $Proposition\ 6$ are generated by a pair of
selfadjoint operators and a single selfadjoint operator, respectively.

We shall interest to describe the class of

$\mathbf{(i)}\ $all invertible operators $S\in \mathfrak{B}(H)$ satisfying \
the operator inequality $(3^{^{\prime }}),$

$\mathbf{(ii)}\ $all operators with closed ranges $S\in \mathfrak{B}(H)$
satisfying \ the operator inequality $(2^{^{\prime }}),$

$\mathbf{(iii)}\ $all operators $S\in \mathfrak{B}(H)$ satisfying \ the
operator inequality $(4^{^{\prime }}).$

We shall prove that the class

\begin{enumerate}
\item[$\bullet $] $\mathbf{(i)}$ is the class of all invertible selfadjoint
operators in $\mathfrak{B}(H)$ multiplied by nonzero scalars$,$

\item[$\bullet $] $\mathbf{(ii)}$ is the class of all selfadjoint operators
with closed ranges in $\mathfrak{B}(H)$ multiplied by nonzero scalars$,$

\item[$\bullet $] $\mathbf{(iii)}$ is the class of all selfadjoint operators
in $\mathfrak{B}(H)$ multiplied by nonzero scalars$..$
\end{enumerate}

We shall present here all these characterizations and others of the class of
all invertible selfadjoint operators multiplied by nonzero scalars, the
class of all selfadjoint operators with closed ranges multiplied by nonzero
scalars, and the class of all selfadjoint operators multiplied by nonzero
scalars,

We need the following lemma.

\begin{lemma}
$[12]\ $Let $\lambda ,\ \mu \in \mathbb{C}^{\ast }$ such that $\frac{\lambda 
}{\mu }+\frac{\mu }{\lambda }\in \mathbb{R}$, and $\left\vert \frac{\lambda 
}{\mu }+\frac{\mu }{\lambda }\right\vert \geq 2$.\ Then there exists $\theta
\in \lbrack 0,\pi )$ such that $\lambda ,\ \mu \in D_{_{\theta }}.$
\end{lemma}

\begin{proof}
Let $\lambda =re^{i\alpha },\ \mu =le^{i\beta }$ be the polar decomposition
of $\lambda ,\ \mu $.\ Then we have 
\begin{equation*}
\frac{\lambda }{\mu }+\frac{\mu }{\lambda }=\left( \frac{r}{l}+\frac{l}{r}%
\right) \cos \left( \alpha -\beta \right) +i\left( \frac{r}{l}-\frac{l}{r}%
\right) \sin \left( \alpha -\beta \right) .
\end{equation*}

Thus, $r=l$ or $\alpha -\beta \equiv 0\ (\func{mod}.\pi )$.\ The case $r=l$
also gives $\alpha -\beta \equiv 0\ (\func{mod}.\pi )$. Hence, the prof is
completed.
\end{proof}

\begin{proposition}
$[12]\ $Let $S$ be an invertible operator in $\mathfrak{B}(H)$. Then the
following properties are equivalent:

$(i)\ S$ is a selfadjoint operator multiplied by a nonzero scalar,

$(ii)\ \forall X\in \mathfrak{B}(H),\ \left\Vert
SXS^{-1}+S^{-1}XS\right\Vert =\left\Vert S^{\ast }XS^{-1}+S^{-1}XS^{\ast
}\right\Vert ,$

$(iii)\ \forall X\in \mathfrak{B}(H),\ \left\Vert
SXS^{-1}+S^{-1}XS\right\Vert \geq \left\Vert S^{\ast }XS^{-1}+S^{-1}XS^{\ast
}\right\Vert ,$

$(iv)\ \forall X\in \mathfrak{B}(H),\ \left\Vert
SXS^{-1}+S^{-1}XS\right\Vert \geq 2\left\Vert X\right\Vert .$
\end{proposition}

\begin{proof}
The two implications $(i)\Longrightarrow (ii)$ and $(ii)\Longrightarrow
(iii) $ are trivial.

The implication $(iii)\Longrightarrow (iv)$ follows from $Corollary$ $%
4.6^{\prime }.$

$(iv)\Rightarrow (i).$ Assume $(iv)$ holds.

So, we have%
\begin{equation*}
\forall X\in \mathfrak{B}(H),\ \left\Vert SXS^{-1}\right\Vert +\left\Vert
S^{-1}XS\right\Vert \geq 2\left\Vert X\right\Vert .
\end{equation*}

Using $Proposition\ 2$, then $S$ is normal. \ Using the spectral measure of $%
S$, there exists a sequence $(S_{_{n}})$ of invertible normal operators with
finite spectrum such that:

$(a)\ S_{_{n}}\rightarrow S$ uniformly,

$(b)$ for all $\lambda \in \sigma (S)$, there exists a sequence $(\lambda
_{_{n}})$ such that $\lambda _{_{n}}\in \sigma (S_{_{n}})$, for all $n$ and $%
\lambda _{_{n}}\rightarrow \lambda .$

Let $\lambda ,\ \mu \in \sigma (S)$ and let $\epsilon >0$. Using $(ii),\
(a),\ $and $(b)$, there exists an integer $N\geq 1$ such that 
\begin{equation*}
(1)\ \ \ \ \forall n>N,\ \forall X\in \mathfrak{B}(H),\ \left\Vert
S_{_{n}}XS_{_{n}}^{-1}+S_{_{n}}^{-1}XS_{_{n}}\right\Vert \geq (2-\epsilon
)\left\Vert X\right\Vert ,
\end{equation*}%
and there exist two sequences $(\lambda _{_{n}}),\ (\mu _{_{n}})$ such that $%
\lambda _{_{n}},\ \mu _{_{n}}\in \sigma (S_{_{n}}),\ $for all $n,$ and $%
\lambda _{_{n}}\rightarrow \lambda ,\ \mu _{_{n}}\rightarrow \mu .$

Let $n>N$ and since $S_{_{n}}$ is normal, with finite spectrum, there exist $%
p$ orthogonal projections $E_{_{1}},...,E_{_{p}}$ in $\mathfrak{B}(H)$ such
that $E_{_{i}}E_{_{j}}=0$, if $i\neq j$, $\ \sum_{i=1}^{p}E_{_{i}}=I,\
S_{_{n}}=\sum_{i=}^{p}\alpha _{_{i}}E_{i}$, where $\sigma (S_{_{n}})=\left\{
\alpha _{_{1}},...,\alpha _{_{p}}\right\} ,\ \alpha _{_{1}}=\lambda
_{_{n}},\ \alpha _{_{2}}=\mu _{_{n}}.$

Then by $(1)$, and if we put $A=\left[ 
\begin{array}{cc}
2 & \gamma _{_{n}} \\ 
\gamma _{_{n}} & 2%
\end{array}%
\right] $, where $\gamma _{_{n}}=\frac{\lambda _{_{n}}}{\mu _{_{n}}}+\frac{%
\mu _{_{n}}}{\lambda _{_{n}}}$, we obtain%
\begin{equation*}
(2)\ \ \ \ \forall X\in \mathfrak{B}(\mathbb{C}^{2}),\ \left\Vert A\circ
X\right\Vert \geq \left( 2-\epsilon \right) \left\Vert X\right\Vert .
\end{equation*}

If we put $\delta _{_{n}}=\frac{1}{\gamma _{_{n}}},$ and $B=\left[ 
\begin{array}{cc}
\frac{1}{2} & \delta _{_{n}} \\ 
\delta _{_{n}} & \frac{1}{2}%
\end{array}%
\right] $, then from the last inequality, we also have%
\begin{equation*}
(3)\ \ \ \ \forall X\in \mathfrak{B}(\mathbb{C}^{2}),\ \left\Vert B\circ
X\right\Vert \leq \frac{\left\Vert X\right\Vert }{\left( 2-\epsilon \right) }%
.
\end{equation*}

From $(2)$, we deduce $\left\vert \frac{\lambda _{_{n}}}{\mu _{_{n}}}+\frac{%
\mu _{_{n}}}{\lambda _{_{n}}}\right\vert \geq 2-\epsilon $.\ Hence, $%
\left\vert \frac{\lambda }{\mu }+\frac{\mu }{\lambda }\right\vert \geq 2$.
Put, $\beta _{_{n}}=\func{Im}\gamma _{_{n}},\ \gamma =\lim \gamma _{_{n}},\
\beta =\lim \beta _{_{n}}.$

On the other hand, if in $(3)$, we put $X=\left[ 
\begin{array}{cc}
1 & ia \\ 
ia & 1%
\end{array}%
\right] $, for $a>0$, we obtain 
\begin{equation*}
\frac{1}{4}+a^{2}\left\vert \gamma _{_{n}}\right\vert ^{2}+a\left\vert \beta
_{_{n}}\right\vert \leq \frac{1+a^{2}}{\left( 2-\epsilon \right) ^{2}}\text{.%
}
\end{equation*}

Hence,%
\begin{equation*}
\frac{1}{4}+a^{2}\left\vert \gamma \right\vert ^{2}+a\left\vert \beta
\right\vert \leq \frac{1+a^{2}}{\left( 2-\epsilon \right) ^{2}}.
\end{equation*}

Thus, $a\left\vert \gamma \right\vert ^{2}+\left\vert \beta \right\vert \leq 
\frac{a}{4}$, for every $a>0.$\ This gives us, $\func{Im}\left( \frac{%
\lambda }{\mu }+\frac{\mu }{\lambda }\right) =\beta =0$.\ So, from the above
lemma, $\lambda $ and $\mu $ belongs to a straight line through the origin.
Then there exists $\theta \in \lbrack 0,\pi )$ such that $\sigma (S)\subset
D_{_{\theta }}$. Therefore, $M=e^{-i\theta }S$ is selfadjoint, and $%
S=e^{i\theta }M$. This proves $(i)$.
\end{proof}

In the next proposition, and from the last proposition concerning the
invertible case, we conclude for the characterization of the class of all
selfadjoint operators in $\mathfrak{B}(H)$\ multiplied by nonzero scalars,

\begin{proposition}
$[20]\ $Let $S\in \mathfrak{B}(H)$. The two following properties are
equivalent

$(i)$ $S$ is a selfadjoint operator multiplied by a nonzero scalar,

$(ii)$ $\forall X\in \mathfrak{B}(H),\ \left\Vert S^{2}X+XS^{2}\right\Vert
\geq 2\left\Vert SXS\right\Vert .$
\end{proposition}

\begin{proof}
We may assume without loss of generality that $\left\Vert S\right\Vert =1.$

$(i)\Rightarrow (ii).$ This implication follows immediately from $(S-AGMI).$

$(ii)\Rightarrow (i).$\ Assume $(ii)$ holds.

Then, we have 
\begin{equation*}
\forall X\in \mathfrak{B}(H),\ \left\Vert S^{2}X\right\Vert +\left\Vert
XS^{2}\right\Vert \geq 2\left\Vert SXS\right\Vert .
\end{equation*}

Hence, from $Proposition\ 4$, $S$ is normal. So, we prove $(i)$ in two cases.

Case 1. $S\in \mathfrak{D}(H)$..

Then, $S$ is invertible. so from $(ii)$, we obtain%
\begin{equation*}
\forall X\in \mathfrak{B}(H),\ \left\Vert SXS^{-1}+S^{-1}XS\right\Vert \geq
2\left\Vert X\right\Vert .
\end{equation*}

Using the last proposition, we deduce $(i).$

Case 2. General situation.

Applying triangular inequality in $(ii)$, we deduce that $\left\Vert
S^{2}\right\Vert =\left\Vert S\right\Vert ^{2}=1.$

Define the real function $F$ on the complete metric space $\left( \mathfrak{B%
}(H)\right) _{1}$ by 
\begin{equation*}
\forall X\in \left( \mathfrak{B}(H)\right) _{1},\ F(X)=\left\Vert
S^{2}X+XS^{2}\right\Vert -2\left\Vert SXS\right\Vert ,
\end{equation*}%
and for $n\geq 1,$ define the real function $F_{n}$ on $\left( \mathfrak{B}%
(H)\right) _{1}$ by 
\begin{equation*}
\forall X\in \mathfrak{B}(H),\ F_{n}(X)=\left\Vert
S_{n}^{2}X+XS_{n}^{2}\right\Vert -2\left\Vert S_{n}XS_{n}\right\Vert .
\end{equation*}

Put $D=\left\{ X\in \left( \mathfrak{B}(H)\right) _{1}:F(X)>0\right\} .$
Then there are two cases, $D=\varnothing $, $D\neq \varnothing .$

$\mathbf{1.}$ $D=\varnothing .$ So, it follows that%
\begin{equation}
\forall X\in \mathfrak{B}(H),\ \left\Vert S^{2}X+XS^{2}\right\Vert
=2\left\Vert SXS\right\Vert .  \tag{$\ast $}
\end{equation}

From this equality, we have%
\begin{equation*}
\forall x,y\in H,\ \left\Vert S^{2}x\otimes y+x\otimes S^{\ast
2}y\right\Vert =2\left\Vert Sx\right\Vert \left\Vert S^{\ast }y\right\Vert .
\end{equation*}

Using this last equality and since $S^{2}\neq 0$, we deduce that $\ker
S^{\ast }=\left\{ 0\right\} $. Hence, $S$ is with dense range. Using again
this last equality, we obtain the following inequality,%
\begin{equation*}
\forall x,y\in \left( H\right) _{1},\ \left\Vert S^{2}x\right\Vert
+2\left\Vert Sx\right\Vert \left\Vert S^{\ast }y\right\Vert \geq \left\Vert
S^{\ast 2}y\right\Vert .
\end{equation*}

By taking the supremum over $y\in \left( H\right) _{1}$, we obtain that $%
\left\Vert Sx\right\Vert \geq \frac{1}{3}\left\Vert x\right\Vert $, for
every $x\in H$. Thus, $S$ is bounded below with dense range. Hence, $S$ is
invertible. So, from $(\ast )$, it follows that 
\begin{equation*}
\forall X\in \mathfrak{B}(H),\ \left\Vert SXS^{-1}+S^{-1}XS\right\Vert
=2\left\Vert X\right\Vert .
\end{equation*}

Then from the last proposition, $(i)$ holds.

$\mathbf{2.}$ $D\neq \varnothing .$ From the fact that $F$ is a positive
continuous map on $\left( \mathfrak{B}(H)\right) _{1}$, it follows that 
\begin{equation*}
\overline{D}=\overline{F^{-1}\left( \left( 0,\infty \right) \right) }%
=F^{-1}\left( [0,\infty )\right) =\left\{ X\in \left( \mathfrak{B}(H)\right)
_{1}:F(X)\geq 0\right\} =\left( \mathfrak{B}(H)\right) _{1}.
\end{equation*}

Let $X\in D$, and $\epsilon >0.$ Since $S_{n}\rightarrow S$ uniformly, then
there exists an integer $N\geq 1$ (depends only in $\epsilon $) such that 
\begin{equation*}
\forall n\geq N,\ \forall Y\in \left( \mathfrak{B}(H)\right) _{1},\
\left\vert F(Y)-F_{n}(Y)\right\vert \leq \epsilon .
\end{equation*}

Using the same argument as used in $Proposition\ 4$, it follows that%
\begin{equation*}
\forall X\in D,\ \forall n\geq N,\ \ F_{n}(X)\geq 0.
\end{equation*}

Since each $F_{n}$ is a continuous map on $\left( \mathfrak{B}(H)\right)
_{1} $ and $D$ is dense in $\left( \mathfrak{B}(H)\right) _{1}$, then%
\begin{equation*}
\forall X\in \left( \mathfrak{B}(H)\right) _{1},\ \forall n\geq N,\
F_{n}(X)\geq 0.
\end{equation*}

So, it follows that%
\begin{equation*}
\forall X\in \mathfrak{B}(H),\ \forall n\geq N,\ \left\Vert
S_{n}^{2}X+XS_{n}^{2}\right\Vert \geq 2\left\Vert S_{n}XS_{n}\right\Vert .
\end{equation*}

Since for each $n\geq 1$, $S_{n}\in \mathfrak{D}(H),$ using the case 1, we
obtain that $S_{n}$ is a selfadjoint with closed range multiplied by nonzero
scalar, for every $n\geq N$. Since $S_{n}$ $\rightarrow S$ uniformly, and
the class of all selfadjoint operators in $\mathfrak{B}(H)$ is closed in $%
\mathfrak{B}(H)$, this proves $(i)$.
\end{proof}

\begin{corollary}
$[18]\ $Let $S$ be an operator with closed range in $\mathfrak{B}(H)$. Then
the following properties are equivalent:

$(i)\ S$ is a selfadjoint operator multiplied by a non zero scalar,

$(ii)\ \forall X\in \mathfrak{B}(H),\ \left\Vert SXS^{+}+S^{+}XS\right\Vert
=\left\Vert S^{\ast }XS^{+}+S^{+}XS^{\ast }\right\Vert ,$

$(iii)\ \forall X\in \mathfrak{B}(H),\ \left\Vert SXS^{+}+S^{+}XS\right\Vert
\geq \left\Vert S^{\ast }XS^{+}+S^{+}XS^{\ast }\right\Vert ,$

$(iv)\ \forall X\in \mathfrak{B}(H),\ \left\Vert SXS^{+}+S^{+}XS\right\Vert
\geq 2\left\Vert SS^{+}XS^{+}S\right\Vert .$
\end{corollary}

\begin{proof}
The implications $(i)\Rightarrow (ii)$ and $(ii)\Rightarrow (iii)$ are
trivial.

The implication $(iii)\Rightarrow (iv)$\ follows immediately from $%
Corollary\ 4.5^{\prime }$.

$(iv)\Rightarrow (i)$.\ Assume $(iv)$ holds. Applying the triangular
inequality in $(iv),$ we obtain from $Corollary\ 3$, that $S$ is normal
(with a closed range). So that $S$ is an EP operator satisfying $(iv)$.
Then, $S=\left[ 
\begin{array}{cc}
S_{_{1}} & 0 \\ 
0 & 0%
\end{array}%
\right] \left[ 
\begin{array}{c}
R(S) \\ 
\ker S^{\ast }%
\end{array}%
\right] ,$ where $S_{_{1}}$ is invertible on $R(S)$. Hence, we obtain the
following inequality%
\begin{equation*}
\forall X\in \mathfrak{B}(R(S)),\ \left\Vert
S_{1}XS_{1}^{-1}+S_{1}^{-1}XS_{1}\right\Vert \geq 2\left\Vert X\right\Vert .
\end{equation*}

Using $Proposition\ 7$ with the Hilbert space $R(S)$, we obtain that $S_{1}$
is a selfadjoint operator in $\mathfrak{B}(R(S))$ multiplied by a nonzero
scalar. This proves $(i)$.
\end{proof}

\section{On the injective norm of the two operators $X\longmapsto
SXS^{-1}+S^{-1}XS$ and $X\longmapsto S^{\ast }XS^{-1}+S^{-1}XS^{\ast }$,
unitary operators, and characterizations}

Let $E$ be a (real or complex) normed space, and let $\mathfrak{B}=\mathfrak{%
B}(E)$ denote the normed algebra of all bounded linear operators acting on $%
E $.

For $A=(A_{_{1}},...,A_{_{n}}),\ B=(B_{_{1}},...,B_{_{n}})$ be two n-tuples
of elements in $\mathfrak{B}$, we define the elementary operator (induced by 
$A$, $B$) $R_{_{A,B}}$ on $\mathfrak{B}$ by:%
\begin{equation*}
\forall X\in \mathfrak{B},\ R_{_{A,B}}(X)=\sum_{i=1}^{n}A_{_{i}}XB_{_{i}}.
\end{equation*}

We denote by $\mathcal{R}(\mathfrak{B})$, the vector space of all elementary
operators on $\mathfrak{B}$. We define the map $d(.):\mathcal{R}(\mathfrak{B}%
)\longrightarrow \mathbb{R}$ by: 
\begin{equation*}
\forall R\in \mathcal{R}(\mathfrak{B}),\ d(R)=\underset{\left\Vert
X\right\Vert =1=rankX}{\sup }\left\Vert R(X)\right\Vert .
\end{equation*}

We consider the tensor product space 
\begin{equation*}
\mathfrak{B}\otimes \mathfrak{B\ }\mathcal{=\ }\left\{
\sum_{i=1}^{n}A_{_{i}}\otimes B_{_{i}}:n\geq 1,\ A_{_{i}},\ B_{_{i}}\in 
\mathfrak{B},\ i=1,...,n\right\} .
\end{equation*}

We denote by $\left\Vert .\right\Vert _{_{\lambda }}$ be the injective norm
on $\mathfrak{B}\otimes \mathfrak{B}$ given by: 
\begin{equation*}
\left\Vert \sum_{i=1}^{n}A_{_{i}}\otimes B_{_{i}}\right\Vert _{_{\lambda }}=%
\underset{f,g\in (\mathfrak{B}^{^{,}})_{_{1}}}{\sup }\left\vert
\sum_{i=1}^{n}f(A_{_{i}})g(B_{_{i}})\right\vert .
\end{equation*}

For $S$ be an invertible operator in $\mathfrak{B}(H)$, we define the two
particular elementary operators $\varphi _{_{S}},\ \psi _{_{S}}$ on $%
\mathfrak{B}(H)$ by%
\begin{equation*}
\left\{ 
\begin{array}{c}
\forall X\in \mathfrak{B}(H),\ \varphi _{_{S}}(X)=SXS^{^{-1}}+S^{^{-1}}XS,
\\ 
\forall X\in \mathfrak{B}(H),\ \psi _{_{S}}(X)=S^{^{\ast
}}XS^{^{-1}}+S^{^{-1}}XS^{^{\ast }}.%
\end{array}%
\right.
\end{equation*}

\begin{notation}
For $A=(A_{_{1}},...,A_{_{n}})$ be an n-tuple of commuting operators in $%
\mathfrak{B}(H)$, we denote by:

1. $\Gamma _{_{A}}$, the set of all multiplicative functional acting on the
maximal commutative Banach algebra that contains the operators $%
A_{_{1}},...,A_{_{n}},$

2. $\sigma (A)=\left\{ \left( \varphi (A_{_{1}}),...,\varphi
(A_{_{n}})\right) :\varphi \in \Gamma _{_{A}}\right\} $, the joint spectrum
of $A$.
\end{notation}

In this section, we shall prove that $\left\Vert
\sum_{i=1}^{n}A_{_{i}}\otimes B_{_{i}}\right\Vert _{_{\lambda }}=d\left(
R_{_{A,B}}\right) $, and where $A=(A_{_{1}},...,A_{_{n}}),\
B=(B_{_{1}},...,B_{_{n}})\in \mathfrak{B}^{n}$, and then the map $d(.)$ is a
norm on the vector space $\mathcal{R}(\mathfrak{B})$, and the two normed
spaces $(\mathcal{R}(\mathfrak{B}),\ d(.))$ and $(\mathfrak{B}\otimes 
\mathfrak{B,\ }\left\Vert .\right\Vert _{_{\lambda }})$ are isometrically
isomorph.

Concerning the characterizations given in section 4 for the invertible case,
they are generated by $\varphi _{_{S}},$ and by $\varphi _{_{S}}$, $\psi
_{_{S}}$ (where $S$ is an invertible operator in $\mathfrak{B}(H)$). It was
proved that the class of all invertible selfadjoint operators in $\mathfrak{B%
}(H)$\ multiplied by nonzero scalars is the class of all invertible
operators $S\in \mathfrak{B}(H)$ satisfying each of the three following
inequalities:%
\begin{equation*}
\begin{array}{cccc}
\mathbf{(i)}\ \forall X\in \mathfrak{B}(H), & \left\Vert \varphi
_{_{S}}(X)\right\Vert & \geq & 2\left\Vert X\right\Vert , \\ 
\mathbf{(ii)}\ \forall X\in \mathfrak{B}(H), & \left\Vert \varphi
_{_{S}}(X)\right\Vert & = & \left\Vert \psi _{_{S}}(X)\right\Vert , \\ 
\mathbf{(iii)}\ \forall X\in \mathfrak{B}(H), & \left\Vert \varphi
_{_{S}}(X)\right\Vert & \geq & \left\Vert \psi _{_{S}}(X)\right\Vert .%
\end{array}%
\end{equation*}

In this section, we consider an invertible operator $S$ in $\mathfrak{B}(H)$.

It is clear that:%
\begin{equation*}
\underset{\left\Vert X\right\Vert =1}{\inf }\left\Vert \varphi
_{_{S}}(X)\right\Vert \leq 2\leq \underset{\left\Vert X\right\Vert =1}{\sup }%
\left\Vert \varphi _{_{S}}(X)\right\Vert .
\end{equation*}

Then, the above infimum gets its maximal value $2$ if and only if $S$
satisfies the condition $\mathbf{(i)}$. So, $\underset{\left\Vert
X\right\Vert =1}{\inf }\left\Vert \varphi _{_{S}}(X)\right\Vert =2$ if and
only if $S$ is an invertible selfadjoint operator\ multiplied by a nonzero
scalar.

In this section, we aboard the problem when the above supremum gets its
minimal value $2.$ We shall show that:

$\mathbf{(iv)}$ this supremum ( that is $\left\Vert \varphi
_{_{S}}\right\Vert $) gets it minimal value $2$ if and only if $S$ is a
unitary operator multiplied by a nonzero scalar,

$\mathbf{(v)}\ \underset{\left\Vert X\right\Vert =1=rankX}{\sup }\left\Vert
\varphi _{_{S}}(X)\right\Vert =\left\Vert S\otimes
S^{^{-1}}+S^{^{-1}}\otimes S\right\Vert _{\lambda }\geq 2,$

$\mathbf{(vi)}$ this last supremum (that is the injective norm of $\varphi
_{_{S}}$) gets its minimal value $2$ if and only if $S$ is normal and $%
\underset{\lambda ,\mu \in \sigma (S)}{\sup }\left\vert \frac{\lambda }{\mu }%
+\frac{\mu }{\lambda }\right\vert =2.$

From above, the following conditions are equivalent:

$\left( \cdot \right) $ $\underset{\left\Vert X\right\Vert =1}{\inf }%
\left\Vert \varphi _{_{S}}(X)\right\Vert =2=\underset{\left\Vert
X\right\Vert =1}{\sup }\left\Vert \varphi _{_{S}}(X)\right\Vert ,$

$\left( \cdot \right) \ \ \forall X\in \mathfrak{B}(H),\ \left\Vert
SXS^{-1}+S^{-1}XS\right\Vert =2\left\Vert X\right\Vert ,$

$\left( \cdot \right) \ S$ is a unitary reflection operator multiplied by a
nonzero scalar.

\begin{proposition}
$[14].\ $For $A=(A_{_{1}},...,A_{_{n}}),\ B=(B_{_{1}},...,B_{_{n}})$ be two
n-tuples of elements in $\mathfrak{B}$, the following equalities hold%
\begin{eqnarray*}
d\left( R_{_{A,B}}\right) &=&\underset{f,g\in (\mathfrak{B}^{^{,}})_{_{1}}}{%
\sup }\left\vert \sum_{i=1}^{n}f(A_{_{i}})g\left( B_{_{i}}\right)
\right\vert , \\
&=&\underset{f\in (\mathfrak{B}^{^{,}})_{_{1}}}{\sup }\left\Vert
\sum_{i=1}^{n}f(B_{_{i}})A_{_{i}}\right\Vert , \\
&=&\underset{f\in (\mathfrak{B}^{^{,}})_{_{1}}}{\sup }\left\Vert
\sum_{i=1}^{n}f(A_{_{i}})B_{_{i}}\right\Vert .
\end{eqnarray*}
\end{proposition}

\begin{proof}
We denote by $k_{_{1}},\ k_{_{2}}\ $and $k_{_{3}}$ be the supremum cited in
the theorem in the same order. Let $x,y\in \left( E\right) _{_{1}},\ h\in
\left( E^{^{\prime }}\right) _{1}$, and let $f,g\in (\mathfrak{B}%
^{^{,}})_{_{1}}$. So, we have%
\begin{eqnarray*}
d\left( R_{_{A,B}}\right) &\geq &\left\Vert \sum_{i=1}^{n}A_{_{i}}\left(
x\otimes h\right) B_{_{i}}y\right\Vert , \\
&=&\left\Vert \left( \sum_{i=1}^{n}h\left( B_{_{i}}y\right) A_{_{i}}\right)
x\right\Vert .
\end{eqnarray*}

By taking the supremum over $x\in (E)_{_{1}}$, we have, $d\left(
R_{_{A,B}}\right) \geq $ $\left\Vert \sum_{i=1}^{n}h\left( B_{_{i}}y\right)
A_{_{i}}\right\Vert $. Thus,%
\begin{eqnarray*}
d\left( R_{_{A,B}}\right) &\geq &\left\vert \sum_{i=1}^{n}f\left(
A_{_{i}}\right) h\left( B_{_{i}}y\right) \right\vert , \\
&=&\left\vert h\left( \sum_{i=1}^{n}f\left( A_{_{i}}\right) B_{_{i}}y\right)
\right\vert .
\end{eqnarray*}%
By taking the supremum over $h\in \left( E^{^{\prime }}\right) _{1}$ and
over $y\in \left( E\right) _{_{1}}$, we obtain 
\begin{equation*}
d\left( R_{_{A,B}}\right) \geq \left\Vert \sum_{i=1}^{n}f\left(
A_{_{i}}\right) B_{_{i}}\right\Vert .
\end{equation*}%
Then, 
\begin{equation*}
d\left( R_{_{A,B}}\right) \geq \left\vert \sum_{i=1}^{n}f\left(
A_{_{i}}\right) g\left( B_{_{i}}\right) \right\vert .
\end{equation*}%
So, we have $d\left( R_{_{A,B}}\right) \geq $ $k_{_{1}}.$ Since, $%
k_{_{1}}\geq \left\vert g\left( \sum_{i=1}^{n}f\left( B_{_{i}}\right)
A_{_{i}}\right) \right\vert $, then $k_{_{1}}\geq \left\Vert
\sum_{i=1}^{n}f\left( B_{_{i}}\right) A_{_{i}}\right\Vert $. This gives us
that $k_{_{1}},\geq k_{_{2}}$. It is clear that $k_{_{2}}\geq \left\vert
g\left( \sum_{i=1}^{n}f\left( A_{_{i}}\right) B_{_{i}}\right) \right\vert $,
then $k_{_{2}}\geq k_{3}$. Since, $k_{3}\geq \left\vert
\sum_{i=1}^{n}f(A_{_{i}})h\left( B_{_{i}}y\right) \right\vert =\left\vert
f\left( \sum_{i=1}^{n}h\left( B_{_{i}}y\right) A_{_{i}}\right) \right\vert $%
, so we have, 
\begin{eqnarray*}
k_{3} &\geq &\left\Vert \sum_{i=1}^{n}h\left( B_{_{i}}y\right)
A_{_{i}}\right\Vert , \\
&\geq &\left\vert \sum_{i=1}^{n}h\left( B_{_{i}}y\right)
A_{_{i}}x\right\vert , \\
&=&\left\Vert \left( \sum_{i=1}^{n}A_{_{i}}\left( x\otimes h\right)
B_{_{i}}\right) y\right\Vert .
\end{eqnarray*}%
Thus, $k_{3}\geq \left\Vert \sum_{i=1}^{n}A_{_{i}}\left( x\otimes h\right)
B_{_{i}}\right\Vert $. Therefore, $k_{3}\geq d\left( R_{_{A,B}}\right) $.
This complete the proof.
\end{proof}

\begin{proposition}
The map $d(.):\mathcal{R}\left( \mathfrak{B}\right) \rightarrow \mathbb{R},\
R\mapsto d(R)$ is a norm on $\mathcal{R}\left( \mathfrak{B}\right) .$
\end{proposition}

\begin{proof}
It is clear that 
\begin{equation*}
d(R)\geq 0,\ d(\lambda R)=\left\vert \lambda \right\vert d(R),\ d(R+S)\leq
d(R)+d(S),
\end{equation*}%
for every scalar $\lambda $, and for every $R,S\in \mathcal{R}\left( 
\mathfrak{B}\right) $. So, it remains to prove that if $d(R)=0$, then $R=0$,
for every $R\in \mathcal{R}\left( \mathfrak{B}\right) $.

Now, let $A=(A_{_{1}},...,A_{_{n}}),\ B=(B_{_{1}},...,B_{_{n}})$ be two
n-tuples of elements in $\mathfrak{B}$ such that $d\left( R_{_{A,B}}\right)
=0$.

We may assume that $B_{_{1}},...,B_{_{m}}$ (where $m\leq n$) form a maximal
linearly independent subset of $B_{_{1}},...,B_{_{n}}$. There exist $m$
operators $C_{_{1}},...,C_{m}\in sp\left\{ A_{_{1}},...,A_{_{n}}\right\} $
such that $R_{_{A,B}}=R_{_{C},_{D}}$, where $C=\left(
C_{_{1}},...,C_{_{m}}\right) ,\ D=\left( B_{_{1}},...,B_{_{m}}\right) $. So,
using the above proposition, we obtain $\sum_{i=1}^{m}f(C_{_{i}})B_{_{i}}=0$%
, for every $f\in \mathfrak{B}^{^{,}}$. Since $B_{_{1}},...,B_{_{n}}$ are
linearly independent, then $f(C_{_{i}})=0$, for $i=1,...,m$, and for every $%
f\in \mathfrak{B}^{^{,}}$. This proves, $C_{_{i}}=0$, for $i=1,...,m$.
Hence, $R_{_{A,B}}=R_{_{C},_{D}}=0.$
\end{proof}

\begin{corollary}
The two normed spaces $\left( \mathcal{R}\left( \mathfrak{B}\right)
,d(.)\right) $ and $\left( \mathfrak{B}\otimes \mathfrak{B},\left\Vert
.\right\Vert _{_{\lambda }}\right) $ are isometrically isomorph.
\end{corollary}

\begin{proof}
Let the map

\begin{equation*}
\begin{array}{ccc}
\Gamma :\left( \mathfrak{B}\otimes \mathfrak{B},\left\Vert .\right\Vert
_{_{\lambda }}\right) & \rightarrow & \left( \mathcal{R}\left( \mathfrak{B}%
\right) ,d(.)\right) , \\ 
\sum_{i=1}^{n}A_{_{i}}\otimes B_{_{i}} & \longmapsto & R_{_{A,B}},%
\end{array}%
\end{equation*}%
where $A=(A_{_{1}},...,A_{_{n}}),\ B=(B_{_{1}},...,B_{_{n}})\in \mathfrak{B}%
^{n}$.

From $Proposition\ 9$, the map $\Gamma $ is well-defined and injective. It
is clear that $\Gamma $ is linear and surjective. Using again $Proposition\
9 $, we deduce that $\Gamma $ is an isometry.
\end{proof}

\begin{notation}
According to the above identification, and for $R\in \mathcal{R}\left( 
\mathfrak{B}\right) $, we use the notation $\left\Vert R\right\Vert
_{_{\lambda }}$ instead of $d(R)$, and we say it is the injective norm of $R$
\end{notation}

\begin{corollary}
$[15].\ $Let $A=(A_{_{1}},...,A_{_{n}}),\ B=(B_{_{1}},...,B_{_{n}})$ be two
n-tuples of commuting operators in $\mathfrak{B}(H)$. Then $\left\Vert
R_{_{A,B}}\right\Vert _{_{\lambda }}\geq \left\vert \sigma (A)\circ \sigma
(B)\right\vert $; and this inequality becomes an equality, if all $A_{_{i}}$
and $B_{_{i}}$ are normal operators.
\end{corollary}

\begin{proof}
Let $(\varphi ,\psi )$ be an arbitrary pair in $\Gamma _{_{A}}\times \Gamma
_{_{B}}$. Using the Hahn-Banach theorem, we may extend $\varphi $ and $\psi $
to unit functional $f$ and $g$ on $\mathfrak{B}(H)$, respectively. So from $%
Proposition\ 9$, it follows that $\left\Vert R_{_{A,B}}\right\Vert
_{_{\lambda }}\geq \left\vert
\sum_{i=1}^{n}f(A_{_{i}})g(B_{_{i}})\right\vert =\left\vert
\sum_{i=1}^{n}\varphi (A_{_{i}})\psi (B_{_{i}})\right\vert $. Therefore $%
\left\Vert R_{_{A,B}}\right\Vert _{_{\lambda }}\geq \left\vert \sigma
(A)\circ \sigma (B)\right\vert $.

Now, suppose all $A_{_{i}}$ and $B_{_{i}}$ are normal operators. It suffice
to prove $\left\Vert R_{_{A,B}}\right\Vert _{_{\lambda }}\leq \left\vert
\sigma (A)\circ \sigma (B)\right\vert $. Let $f,\ g$ be two arbitrary unit
functional on $\mathfrak{B}(H)$, and let $(\varphi ,\psi )$ be an arbitrary
pair in $\Gamma _{_{A}}\times \Gamma _{_{B}}$. Since $\left\vert \sigma
(A)\circ \sigma (B)\right\vert \geq \left\vert \psi \left(
\sum_{i=1}^{n}\varphi (A_{_{i}})B_{_{i}}\right) \right\vert $, and $%
\sum_{i=1}^{n}\varphi (A_{_{i}})B_{_{i}}$ is normal (from Putnam-Fuglede),
then $\left\vert \sigma (A)\circ \sigma (B)\right\vert \geq \left\Vert
\sum_{i=1}^{n}\varphi (A_{_{i}})B_{_{i}}\right\Vert $. So that $\left\vert
\sigma (A)\circ \sigma (B)\right\vert \geq \left\Vert \sum_{i=1}^{n}\varphi
(A_{_{i}})g\left( B_{_{i}}\right) \right\Vert =\left\Vert \varphi \left(
\sum_{i=1}^{n}g\left( B_{_{i}}\right) A_{_{i}}\right) \right\Vert $. Using
the same argument as used with $B_{_{i}}$, we deduce that $\left\vert \sigma
(A)\circ \sigma (B)\right\vert \geq \left\Vert \sum_{i=1}^{n}g\left(
B_{_{i}}\right) A_{_{i}}\right\Vert $. From $Proposition\ 9$, it follows
that $\left\vert \sigma (A)\circ \sigma (B)\right\vert \geq $ $\left\Vert
R_{_{A,B}}\right\Vert _{_{\lambda }}.$
\end{proof}

\begin{lemma}
We have $\left\Vert \psi _{_{S}}\right\Vert _{_{\lambda }}=\left\Vert
\varphi _{_{P}}\right\Vert _{_{\lambda }}$, where $P=\left\vert S\right\vert
.$
\end{lemma}

\begin{proof}
Let $S=UP$, be the polar decomposition of $S$.\ From the fact that 
\begin{equation*}
\left\{ X\in \mathfrak{B}(H):\left\Vert X\right\Vert =1=rankX\right\}
=\left\{ U^{\ast }X:X\in \mathfrak{B}(H),\ \left\Vert X\right\Vert
=1=rankX\right\} ,
\end{equation*}%
it follows that%
\begin{eqnarray*}
\left\Vert \psi _{_{S}}\right\Vert _{_{\lambda }} &=&\underset{\left\Vert
X\right\Vert =1=rankX}{\sup }\left\Vert S^{^{\ast }}XS^{-1}+S^{-1}XS^{^{\ast
}}\right\Vert \\
&=&\underset{\left\Vert X\right\Vert =1=rankX}{\sup }\left\Vert PU^{^{\ast
}}XP^{-1}U^{^{\ast }}+P^{-1}U^{^{\ast }}XPU^{^{\ast }}\right\Vert \\
&=&\underset{\left\Vert X\right\Vert =1=rankX}{\sup }\left\Vert P\left(
U^{^{\ast }}X\right) P^{-1}+P^{-1}\left( U^{^{\ast }}X\right) P\right\Vert \\
&=&\left\Vert \varphi _{_{P}}\right\Vert _{_{\lambda }}.
\end{eqnarray*}
\end{proof}

\begin{proposition}
$[15].\ $The following properties hold

\begin{equation*}
(i).\left\Vert \varphi _{_{S}}\right\Vert _{_{\lambda }}\geq \underset{%
\lambda ,\mu \in \sigma (S)}{\sup }\left\vert \frac{\lambda }{\mu }+\frac{%
\mu }{\lambda }\right\vert ,
\end{equation*}

$(ii).\ $if $S$ is normal, the above inequality becomes an equality,

$(iii).\ $if $S$ is normal,, the following holds%
\begin{equation*}
\left\Vert \psi _{_{S}}\right\Vert _{_{\lambda }}=\underset{\lambda ,\mu \in
\sigma (S)}{\sup }\left( \left\vert \frac{\lambda }{\mu }\right\vert
+\left\vert \frac{\mu }{\lambda }\right\vert \right) .
\end{equation*}
\end{proposition}

\begin{proof}
$(i)$ and $(ii)$ follows immediately from $Corollary$ $7$.

$(iii).$ Assume $S$ normal, and let $UP$ be its polar decomposition.

Since $S$ is invertible and normal , then $\sigma (P)=\left\{ \left\vert
\lambda \right\vert :\lambda \in \sigma (S)\right\} $. So from the above
lemma and $(ii)$, we obtain $\left\Vert \psi _{_{S}}\right\Vert _{_{\lambda
}}=\left\Vert \varphi _{_{P}}\right\Vert _{_{\lambda }}=$ $\underset{\lambda
,\mu \in \sigma (S)}{\sup }\left( \left\vert \frac{\lambda }{\mu }%
\right\vert +\left\vert \frac{\mu }{\lambda }\right\vert \right) .$
\end{proof}

\begin{corollary}
$(i).$ We have $\left\Vert \varphi _{_{S}}\right\Vert _{_{\lambda }}\geq 2,$

$(ii).$ if $S$ is normal, then the injective norm of $\varphi _{_{S}}$ gets
its minimal value $2$, if and only if the following spectral condition holds 
\begin{equation*}
\forall \in \lambda ,\mu \in \sigma (S),\ \left\vert \frac{\lambda }{\mu }+%
\frac{\mu }{\lambda }\right\vert \leq 2.
\end{equation*}

$(iii).$if $\left\Vert \varphi _{_{S}}\right\Vert _{_{\lambda }}=2$, then
the interior of $\sigma (S)$ is empty.
\end{corollary}

\begin{proof}
$(i)$ and $(ii)$ follows immediately from the above proposition.

$(iii)$ Assume $\left\Vert \varphi _{_{S}}\right\Vert _{_{\lambda }}=2$.
Thus, $\left\vert \frac{\lambda }{\mu }+\frac{\mu }{\lambda }\right\vert
\leq 2$, for every $\lambda ,\mu \in \sigma (S)$. Hence, every straight line
passing through the origin intercept $\sigma (S)$ in at most two points.
This proves that the interior of $\sigma (S)$ is empty.
\end{proof}

\begin{proposition}
$[15].\ $Let $P$ be a positive and invertible operator in $\mathfrak{B}(H)$.
Then we have, 
\begin{equation*}
\left\Vert \varphi _{_{P}}\right\Vert _{_{\lambda }}=\left\Vert P\right\Vert
\left\Vert P^{-1}\right\Vert +\frac{1}{\left\Vert P\right\Vert \left\Vert
P^{-1}\right\Vert }.
\end{equation*}
\end{proposition}

\begin{proof}
Let the operator $M_{_{P}}$ defined on $\mathfrak{B}(H)$ by 
\begin{equation*}
\forall X\in \mathfrak{B}(H),\ M_{_{P}}(X)=PXP^{-1}.\ 
\end{equation*}

Since $\sigma (M_{_{P}})=\sigma (P)\sigma (P^{-1})$, $\sigma (\varphi
_{_{P}})=\left\{ f(M_{_{P}})+\frac{1}{f(M_{_{P}})}:f\in \Gamma \right\} $
(where $\Gamma $ is the set of all multiplicative functional on the maximal
commutative Banach algebra in $\mathfrak{B}\left( \mathfrak{B}(H)\right) $
that contains $M_{_{P}}$), and from the above proposition, it is easy to see
that, $\left\Vert \varphi _{_{P}}\right\Vert _{_{\lambda }}=\underset{%
\lambda ,\mu \in \sigma (P)}{\sup }\left\vert \frac{\lambda }{\mu }+\frac{%
\mu }{\lambda }\right\vert =\underset{z\in \sigma (M_{_{P}})}{\sup }%
\left\vert z+\frac{1}{z}\right\vert $. So, using the fact that $\min \sigma
(P)=\frac{1}{\left\Vert P^{-1}\right\Vert }$ and $\max \sigma (P)=\left\Vert
P\right\Vert $, then $\min \sigma (M_{_{P}})=\frac{1}{\left\Vert
P\right\Vert \left\Vert P^{-1}\right\Vert }=p$, and $\max \sigma
(M_{_{P}})=\left\Vert P\right\Vert \left\Vert P^{-1}\right\Vert =\frac{1}{p}$%
.\ On the other hand, since $\underset{p\leq t\leq \frac{1}{p}}{\max }\left(
t+\frac{1}{t}\right) =p+\frac{1}{p}$, this maximum is attainable at $p$ and $%
\frac{1}{p}$. Thus, the result follows immediately from the fact that $p\in
\sigma (M_{_{P}})$.
\end{proof}

\begin{proposition}
$[15]\ $The following properties hold:

$(i)\ \left\Vert \psi _{_{S}}\right\Vert _{_{\lambda }}=\left\Vert
S\right\Vert \left\Vert S^{^{-1}}\right\Vert +\frac{1}{\left\Vert
S\right\Vert \left\Vert S^{^{-1}}\right\Vert },$

$(ii)\ $If $S$ is selfadjoint, then $\left\Vert \varphi _{_{S}}\right\Vert
_{_{\lambda }}=\left\Vert S\right\Vert \left\Vert S^{^{-1}}\right\Vert +%
\frac{1}{\left\Vert S\right\Vert \left\Vert S^{^{-1}}\right\Vert },$

$(iii)\ $if $S$ normal, then $\left\Vert \varphi _{_{S}}\right\Vert
_{_{\lambda }}\leq \left\Vert S\right\Vert \left\Vert S^{^{-1}}\right\Vert +%
\frac{1}{\left\Vert S\right\Vert \left\Vert S^{^{-1}}\right\Vert }.$
\end{proposition}

\begin{proof}
Let $S=UP$ be the polar decomposition of $S$.

$(i).$ From $Lemma\ 7$ and $Proposition\ 12$\ and since, $\left\Vert
S\right\Vert =\left\Vert P\right\Vert ,\ \left\Vert S^{^{-1}}\right\Vert
=\left\Vert P^{^{-1}}\right\Vert $, it follows that:%
\begin{equation*}
\left\Vert \psi _{_{S}}\right\Vert _{_{\lambda }}=\left\Vert \varphi
_{_{P}}\right\Vert _{_{\lambda }}=\left\Vert P\right\Vert \left\Vert
P^{^{-1}}\right\Vert +\frac{1}{\left\Vert P\right\Vert \left\Vert
P^{^{-1}}\right\Vert }=\left\Vert S\right\Vert \left\Vert
S^{^{-1}}\right\Vert +\frac{1}{\left\Vert S\right\Vert \left\Vert
S^{^{-1}}\right\Vert }.
\end{equation*}

$(ii).$ This implication follows immediately from $(i)$.

$(iii)$. Assume $S$ normal.\ Then, using $Proposition\ 11$, and $(i)$, we
obtain 
\begin{eqnarray*}
\left\Vert \varphi _{_{S}}\right\Vert _{_{\lambda }} &=&\underset{\lambda
,\mu \in \sigma (S)}{\sup }\left\vert \frac{\lambda }{\mu }+\frac{\mu }{%
\lambda }\right\vert \\
&\leq &\underset{\lambda ,\mu \in \sigma (S)}{\sup }\left( \left\vert \frac{%
\lambda }{\mu }\right\vert +\left\vert \frac{\mu }{\lambda }\right\vert
\right) \\
&=&\left\Vert \psi _{_{S}}\right\Vert _{_{\lambda }} \\
&=&\left\Vert S\right\Vert \left\Vert S^{^{-1}}\right\Vert +\frac{1}{%
\left\Vert S\right\Vert \left\Vert S^{^{-1}}\right\Vert }.\ 
\end{eqnarray*}
\end{proof}

\begin{remark}
In the condition $(iii)$ of this last proposition, the inequality may be
strict. Indeed, in dimension two, we choose the invertible normal operator $%
S=\left[ 
\begin{array}{cc}
1 & 0 \\ 
0 & \frac{1+i}{2}%
\end{array}%
\right] $. By a simple computation, we find that $2=\left\Vert \varphi
_{_{S}}\right\Vert _{_{\lambda }}<\left\Vert S\right\Vert \left\Vert
S^{^{-1}}\right\Vert +\frac{1}{\left\Vert S\right\Vert \left\Vert
S^{^{-1}}\right\Vert }=\frac{3\sqrt{2}}{2}.$
\end{remark}

\begin{notation}
We denote by 
\begin{equation*}
\mathcal{E}(H)\ \mathcal{=}\left\{ T\in \mathfrak{B}(H):T\ \text{normal and
invertible, }\left\Vert \varphi _{_{T}}\right\Vert _{_{\lambda }}=\left\Vert
T\right\Vert \left\Vert T^{^{-1}}\right\Vert +\frac{1}{\left\Vert
T\right\Vert \left\Vert T^{^{-1}}\right\Vert }\right\} .
\end{equation*}
\end{notation}

From above, $\mathcal{E}(H)$ contains every invertible selfadjoint (resp.
every unitary) operators in $\mathfrak{B}(H)$, but $\mathcal{E}(H)$ does not
contain every invertible normal operators in $\mathfrak{B}(H)\ $(see the
example in the above remark). In the next proposition, we give a
characterization of this class $\mathcal{E}(H)$, where we use the following
notations:

\begin{enumerate}
\item[$\bullet $] $\ \sigma _{_{1}}(S)=\left\{ \lambda \in \sigma
(S):\left\vert \lambda \right\vert =\underset{\mu \in \sigma (S)}{\min }%
\left\vert \mu \right\vert \right\} ,\ $

\item[$\bullet $] $\ \sigma _{_{2}}(S)=\left\{ \lambda \in \sigma
(S):\left\vert \lambda \right\vert =r(S)\right\} ,$

\item[$\bullet $] $D_{_{\theta }}$ (where $\theta \in \lbrack 0,\pi )$) is
the straight line through the origin with slope $\tan \theta $.
\end{enumerate}

\begin{proposition}
$[15]\ $The two following properties are equivalent:

$(i)$ $S\in \mathcal{E}(H)$,

$(ii)$ $S$ is normal, and there exists $\theta \in \lbrack 0,\pi \lbrack $
such that 
\begin{equation*}
D_{_{\theta }}\cap \sigma _{_{1}}(S)\neq \emptyset ,\ D_{_{\theta }}\cap
\sigma _{2}(S)\neq \emptyset .
\end{equation*}
\end{proposition}

\begin{proof}
$(i)\Rightarrow (ii).$ Assume $(i)$ holds. Using $Proposition\ 11$ and from
the compactness of $\sigma (S)$, we may choose $\lambda ,\mu \in \sigma (S)$
such that 
\begin{equation*}
\left\Vert S\right\Vert \left\Vert S^{^{-1}}\right\Vert +\frac{1}{\left\Vert
S\right\Vert \left\Vert S^{^{-1}}\right\Vert }=\left\Vert \varphi
_{_{S}}\right\Vert _{_{\lambda }}=\left\vert \frac{\lambda }{\mu }+\frac{\mu 
}{\lambda }\right\vert .
\end{equation*}%
Hence, 
\begin{eqnarray*}
\left\Vert S\right\Vert \left\Vert S^{^{-1}}\right\Vert +\frac{1}{\left\Vert
S\right\Vert \left\Vert S^{^{-1}}\right\Vert } &\leq &\left\vert \frac{%
\lambda }{\mu }\right\vert +\left\vert \frac{\mu }{\lambda }\right\vert , \\
&\leq &\left\Vert \psi _{_{S}}\right\Vert _{_{\lambda }}, \\
&=&\left\Vert S\right\Vert \left\Vert S^{^{-1}}\right\Vert +\frac{1}{%
\left\Vert S\right\Vert \left\Vert S^{^{-1}}\right\Vert }.
\end{eqnarray*}%
Thus, $\left\vert \frac{\lambda }{\mu }+\frac{\mu }{\lambda }\right\vert
=\left\vert \frac{\lambda }{\mu }\right\vert +\left\vert \frac{\mu }{\lambda 
}\right\vert =\left\Vert S\right\Vert \left\Vert S^{^{-1}}\right\Vert +\frac{%
1}{\left\Vert S\right\Vert \left\Vert S^{^{-1}}\right\Vert }.$ Put $p=\frac{1%
}{\left\Vert S\right\Vert \left\Vert S^{^{-1}}\right\Vert }$. Since $S$ is
normal, then $\underset{\lambda ,\mu \in \sigma (S)}{\min }\left\vert \frac{%
\lambda }{\mu }\right\vert =p$, and $\underset{\lambda ,\mu \in \sigma (S)}{%
\max }\left\vert \frac{\lambda }{\mu }\right\vert =\frac{1}{p}$. The
positive function $f(t)=t+\frac{1}{t}$, $\ p\leq t\leq \frac{1}{p}$, is
bounded and attain its maximum $p+\frac{1}{p}=\left\Vert S\right\Vert
\left\Vert S^{^{-1}}\right\Vert +\frac{1}{\left\Vert S\right\Vert \left\Vert
S^{^{-1}}\right\Vert }$ only at $t=$ $p$ and in $t=\frac{1}{p}$. So, we may
choose $\lambda $ in $\sigma _{_{1}}(S)$ and $\mu $ in $\sigma _{_{2}}(S)$.
Since, $\left\vert \frac{\lambda }{\mu }+\frac{\mu }{\lambda }\right\vert
=\left\vert \frac{\lambda }{\mu }\right\vert +\left\vert \frac{\mu }{\lambda 
}\right\vert $, then, $\lambda $ and $\mu $ must be belong to a straight
line passing through the origin. This proves $(ii)$.

$(ii)\Rightarrow (i).$ Assume $(ii)$ holds. Let $\alpha \in D_{_{\theta
}}\cap \sigma _{_{1}}(S)$ and $\beta \in D_{_{\theta }}\cap \sigma _{2}(S)$.
Since, $S$ is normal, then $\alpha =\frac{e^{i\theta }}{\left\Vert
S^{^{-1}}\right\Vert }$ and $\beta =e^{i(\theta +k\pi )}\left\Vert
S\right\Vert $, for some $k\in \left\{ 0,1\right\} .$ Thus, $\left\Vert
\varphi _{_{S}}\right\Vert _{_{\lambda }}\geq \left\vert \frac{\alpha }{%
\beta }+\frac{\beta }{\alpha }\right\vert =\left\Vert S\right\Vert
\left\Vert S^{^{-1}}\right\Vert +\frac{1}{\left\Vert S\right\Vert \left\Vert
S^{^{-1}}\right\Vert }$. Then, using $Proposition\ 13.(iii)$, $(i)$ holds.
\end{proof}

In the next proposition, we shall give two necessary and sufficient
conditions for which $\left\Vert \varphi _{_{S}}\right\Vert _{_{\lambda }}$
gets its minimal value $2$.

We need the two following lemmas:

\begin{lemma}
$[21].$ If $\left\vert \left\langle Sx,x\right\rangle \right\vert \leq 1$
and $\left\vert \left\langle S^{^{-1}}x,x\right\rangle \right\vert \leq 1$,
for every unit vector $x$ in $H$, then $S$ is unitary.
\end{lemma}

\begin{lemma}
The operator $S$ is normal if and only if $S^{^{\ast }}S^{^{-1}}$ is unitary
\end{lemma}

\begin{proof}
The proof is trivial.
\end{proof}

\begin{proposition}
$[16]\ $The following properties are equivalent

$(i)$ $\left\Vert \varphi _{_{S}}\right\Vert _{_{\lambda }}$ gets its
minimal value $2,$

$(ii)$ $\forall X\in \mathcal{F}_{_{1}}(H),\ \left\Vert
SXS^{^{-1}}+S^{^{-1}}XS\right\Vert \leq 2\left\Vert X\right\Vert ,$

$(iii)\ S$ is normal and $\underset{\lambda ,\mu \in \sigma (S)}{\sup }%
\left\vert \frac{\lambda }{\mu }+\frac{\mu }{\lambda }\right\vert =2.$
\end{proposition}

\begin{proof}
$(i)\Leftrightarrow (ii)$. This equivalence follows immediately from $%
Proposition\ 9$ and $Corollary\ 8.(i)$.

$(i)\Rightarrow (iii).$ Assume $(i)$ holds.

From $Proposition\ 11.(i)$, we deduce that $\underset{\lambda ,\mu \in
\sigma (S)}{\sup }\left\vert \frac{\lambda }{\mu }+\frac{\mu }{\lambda }%
\right\vert =2$.

So, it remains to prove that $S$ is normal. By using the same argument as
used in $[13,\ Lemma\ 1]$, we deduce the following inequality%
\begin{equation*}
\forall x,y\in \left( H\right) _{_{1}},\ \left\Vert \varphi
_{_{S}}\right\Vert _{_{\lambda }}\geq 2\left\vert \left\langle
Sx,y\right\rangle \left\langle S^{^{-1}}x,y\right\rangle \right\vert .
\end{equation*}

Hence, the inequality $\left\vert \left\langle Sx,y\right\rangle
\left\langle S^{^{-1}}x,y\right\rangle \right\vert \leq \left\Vert
x\right\Vert \left\Vert y\right\Vert $ holds for every $x,\ y$ in $H$. So we
obtain $\left\vert \left\langle S^{^{\ast }}S^{^{-1}}x,x\right\rangle
\right\vert \leq 1$ and $\left\vert \left\langle \left( S^{^{\ast
}}S^{^{-1}}\right) ^{^{-1}}x,x\right\rangle \right\vert \leq 1$, for every $%
x,\ y$ in $\left( H\right) _{_{1}}$.$\ $Then, using the two above lemmas, we
deduce that $S$ is normal. Thus, $(iii)$ holds.

$(iii)\Rightarrow (i).$ This follows immediately from $Proposition\ 11.(ii)$.
\end{proof}

\begin{remark}
The class of all operators $S$ for which $\left\Vert \varphi
_{_{S}}\right\Vert _{_{\lambda }}$ is minimal contains strictly the class of
all unitary operators, and contained strictly in the class of all invertible
normal operators.\ Indeed, it is easy to see that $\left\Vert \varphi
_{_{S}}\right\Vert _{_{\lambda }}=2,\ $if $S$ is unitary, and for an
operator $I_{_{1}}\oplus \left( \frac{1}{2}iI_{2}\right) $ with respect to a
some orthogonal direct sum $H=H_{_{1}}\oplus H_{2}$ (where $I_{_{i}}$ is the
identity on $H_{_{i}}$, for $i=1,2$) belongs to this class, but it is not
unitary; for the second inclusion is trivial.
\end{remark}

In the above proposition, we have given two characterizations for which the
injective norm of $\varphi _{_{S}}$ gets its minimal value $2$. In the next
proposition, we shall present some characterizations for which the norm of $%
\varphi _{_{S}}$ gets its minimal value $2$.

\begin{proposition}
$[16]\ $The following properties are equivalent

$(i)\ \forall X\in \mathfrak{B}(H),\ \left\Vert SXS^{^{-1}}\right\Vert
+\left\Vert S^{^{-1}}XS\right\Vert =2\left\Vert X\right\Vert ,$

$(ii)\ \forall X\in \mathfrak{B}(H),\ \left\Vert SXS^{^{-1}}\right\Vert
+\left\Vert S^{^{-1}}XS\right\Vert \leq 2\left\Vert X\right\Vert ,$

$(iii)\ \forall X\in \mathfrak{B}(H),\ \left\Vert
SXS^{^{-1}}+S^{^{-1}}XS\right\Vert \leq 2\left\Vert X\right\Vert ,$

$(iv)\ \left\Vert \varphi _{_{S}}\right\Vert $ gets its minimal value $2,$

$(v)\ \frac{1}{\left\Vert S\right\Vert }S$ is unitary.
\end{proposition}

\begin{proof}
The two implications $(i)\Longrightarrow (ii)$, $(ii)\Longrightarrow (iii)$,
and the equivalence $(iii)\Longleftrightarrow (iv)$ are trivial.

$(iii)\Longrightarrow (v)$. Assume $(iii)$ holds.

So, it follows that, $\left\Vert \varphi _{_{S}}\right\Vert _{_{\lambda }}=2$%
.\ Using $Proposition$ $15$, then $S$ is normal and $\underset{\lambda ,\mu
\in \sigma (S)}{\sup }\left\vert \frac{\lambda }{\mu }+\frac{\mu }{\lambda }%
\right\vert =2.$

Using the spectral measure of $S$, there exists a sequence $(S_{n})$ of
invertible normal operators in $\mathfrak{B}(H)$ with finite spectrum such
that:

$(a)$ $S_{n}\longrightarrow S$ uniformly,

$(b)$ for every $\lambda \in \sigma (S)$, there exists a sequence $(\lambda
_{_{n}})$ such that $\lambda _{_{n}}\in \sigma (S_{n})$, for every $n$, and $%
\lambda _{_{n}}\longrightarrow \lambda $.

Let $\lambda ,\ \mu \in \sigma (S)$. Then from $(b)$, there exist two
sequences $(\lambda _{_{n}}),\ (\mu _{_{n}})$ such that $\lambda _{_{n}},\
\mu _{_{n}}\in \sigma (S_{n})$, for every $n$, and $\lambda
_{_{n}}\longrightarrow \lambda ,\ \mu _{_{n}}\longrightarrow \mu $.

Let $\epsilon >0$. Then, there exists an integer $N\geq 1$ such that 
\begin{equation*}
(\ast )\ \ \ \forall n>N,\ \forall X\in \mathfrak{B}(H),\ \left\Vert
S_{_{n}}XS_{_{n}}^{-1}+S_{_{n}}^{-1}XS_{_{n}}\right\Vert \leq \left(
2+\epsilon \right) \left\Vert X\right\Vert .
\end{equation*}

Let $n>N$. Since $S_{n}$ is normal with finite spectrum, there exist $p$
orthogonal projections $E_{_{1}},...,E_{p}$ in $\mathfrak{B}(H)$ such that $%
E_{_{i}}E_{_{j}}=0$, if $i\neq j$, $\ \sum_{i=1}^{p}E_{_{i}}=I,\
S_{_{n}}=\sum_{i=}^{p}\alpha _{_{i}}E_{i}$, where $\sigma (S_{n})=\left\{
\alpha _{_{1}},...,\alpha _{_{p}}\right\} $, $\alpha _{_{1}}=\lambda _{_{n}}$%
, $\mu _{_{n}}=\alpha _{2}.$

Then, using $(\ast )$ and putting $A=\left[ 
\begin{array}{cc}
2 & \gamma _{n} \\ 
\gamma _{n} & 2%
\end{array}%
\right] $, where $\gamma _{n}=\frac{\lambda _{_{n}}}{\mu _{_{n}}}+\frac{\mu
_{_{n}}}{\lambda _{_{n}}}$, we obtain%
\begin{equation*}
\forall X\in \mathfrak{B}(\mathbb{C}^{2}),\ \left\Vert A\circ X\right\Vert
\leq \left( 2+\epsilon \right) \left\Vert X\right\Vert .
\end{equation*}

Put $X=\left[ 
\begin{array}{cc}
t\func{Im}\gamma _{n} & i \\ 
i & t\func{Im}\gamma _{n}%
\end{array}%
\right] $ (where $t>0$) in this last inequality, we obtain%
\begin{equation*}
\left( 2t\func{Im}\gamma _{n}\right) ^{2}+\left\vert \gamma _{n}\right\vert
^{2}+4t\left( \func{Im}\gamma _{n}\right) ^{2}\leq \left( 2t\func{Im}\gamma
_{n}\right) ^{2}+4+\left( 4\epsilon +\epsilon ^{2}\right) \left( \left( t%
\func{Im}\gamma _{n}\right) ^{2}+1\right) .
\end{equation*}

Put $\gamma =\lim \gamma _{n}=\frac{\lambda }{\mu }+\frac{\mu }{\lambda }$,
and letting $n\longrightarrow \infty $ in this last inequality, it follows
that 
\begin{equation*}
\left\vert \gamma \right\vert ^{2}+4t\left( \func{Im}\gamma \right) ^{2}\leq
4+\left( 4\epsilon +\epsilon ^{2}\right) \left( \left( t\func{Im}\gamma
\right) ^{2}+1\right) .
\end{equation*}

Now, letting $\epsilon \longrightarrow 0$, we deduce that $4t\left( \func{Im}%
\gamma \right) ^{2}\leq 4-\left\vert \gamma \right\vert ^{2}$, for every $%
t>0 $.\ Hence, $\func{Im}\gamma =0$, and $\left\vert \gamma \right\vert \leq
2$. Then, by a simple computation, we find that $\left\vert \lambda
\right\vert =\left\vert \mu \right\vert $. Then $\sigma (S)$ is included in
the circle centred at the origin and of radius $\left\Vert S\right\Vert $.
Since $S$ is normal, this proves $(v)$.

$(v)\Longrightarrow (i).$ This implication is trivial.
\end{proof}

\begin{corollary}
Then the following properties are equivalent:

$(i)\ S$ is a unitary reflection operator multiplied by a nonzero scalar

$(ii)\ \forall X\in \mathfrak{B}(H),\ \left\Vert
SXS^{-1}+S^{-1}XS\right\Vert =2\left\Vert X\right\Vert ,$

$(iii)\ \underset{\left\Vert X\right\Vert =1}{\inf }\left\Vert \varphi
_{_{S}}(X)\right\Vert =2=\underset{\left\Vert X\right\Vert =1}{\sup }%
\left\Vert \varphi _{_{S}}(X)\right\Vert .$
\end{corollary}

\begin{proof}
This corollary follows immediately from $Proposition\ 7$ and $Proposition\
16 $.
\end{proof}

\end{document}